\documentclass[final]{siamltex}

\usepackage{latexsym, enumerate}
\usepackage{eepic}
\usepackage{epic}
\usepackage{graphicx}
\usepackage{color}
\usepackage{ifpdf}
\usepackage{amssymb,amsmath,epsfig, algorithm,algorithmicx,multirow}
\usepackage{amsfonts, dsfont}
\usepackage{subfigure}
\usepackage{bm}

\newtheorem{remark}{Remark}[section]

\title{ Online multiscale  model reduction for nonlinear stochastic PDEs with multiplicative noise
	\thanks{L.Jiang acknowledges the support of NSFC 11871378, the
Fundamental Research Funds for the Central Universities and the support by Shanghai Science and Technology Committee 20JC1413500. }
}

\author{	Lijian Jiang\textsuperscript{1}\thanks{School of Mathematical Sciences,  Tongji University, Shanghai 200092, China. ({\tt  ljjiang@tongji.edu.cn}).}
	\and
	Mengnan Li\textsuperscript{1}\thanks{{\color{black}School of Mathematics}, Hunan University, Changsha 410082, China. ({\tt mengnanli@hnu.edu.cn}).}
	\and	
Meng Zhao\textsuperscript{1}\thanks {School of Mathematical Sciences,  Tongji University, Shanghai 200092, China. ({\tt  1910736@tongji.edu.cn}).}
}

\begin{document}

\maketitle

\begin{abstract}
In this paper, an online multiscale  model reduction method is presented for stochastic partial differential equations (SPDEs) with multiplicative noise, where the diffusion coefficient is
spatially multiscale and the noise perturbation nonlinearly depends on the diffusion dynamics.   It is necessary to efficiently compute all
possible trajectories of the stochastic dynamics for quantifying model's uncertainty and statistic moments.  The multiscale diffusion and nonlinearity may cause the computation
very   intractable.  To overcome the multiscale difficulty,  a constraint energy minimizing generalized multiscale finite element method (CEM-GMsFEM) is used to localize the computation and
obtain an effective coarse model. However, the nonlinear terms are still defined on a fine scale space after  the Galerkin projection of CEM-GMsFEM  is applied to the nonlinear SPDEs.
This significantly impacts on the simulation efficiency  by CEM-GMsFEM.  To this end, a stochastic online discrete empirical interpolation method (DEIM) is proposed to treat the stochastic nonlinearity.
The stochastic online DEIM incorporates offline snapshots and online snapshots.   The offline snapshots consist of the nonlinear terms at the approximate mean of the stochastic dynamics and are used to
construct an offline   reduced model. The online snapshots contain  some information of the current new trajectory  and are used to correct
the offline reduced model in an increment manner.   The stochastic online DEIM substantially reduces the dimension of the nonlinear dynamics and enhances the prediction accuracy for the reduced model.
Thus,   the online multiscale  model reduction is constructed  by using CEM-GMsFEM and the stochastic online DEIM.   A  priori error analysis is carried out  for the nonlinear SPDEs.
 We present a few numerical examples with diffusion in heterogeneous porous media and show the effectiveness of the proposed model reduction.

\end{abstract}



\pagestyle{myheadings}

\thispagestyle{plain}
\markboth{L. Jiang, M. Li and M. Zhao}{Online multiscale model reduction for nonlinear SPDEs}


\section{Introduction}
Fluid flow in porous media has been an active and attractive research field for the last several decades. Contamination of groundwater, underground oil flow in the petroleum industry and blood flow through capillaries are the  relevant porous media applications.  The study of flows in porous media is of significant importance for ecology, industry, biology, etc. These problems  typically involves multiple  spatial and temporal scales. Multiscale coefficients in these models  are often used to describe the porosity, permeability, diffusion process and so on.  The deterministic multiscale models have been thoroughly studied in lot of literatures \cite{homo1,  porous3,jev07,porous4,porous1,porous2} and reference therein.

 However, many natural phenomena in porous media exhibit inherent randomness, such as the permeability of porous media varying  in an irregular manner \cite{ok}. These phenomena can not be modeled with deterministic PDEs. Therefore, more sophisticated theories and concepts are  needed to take account of   the complex behavior of these systems. Stochastic partial differential equations (SPDEs) can  model those natural systems in a comprehensive manner. Stochastic models in porous media \cite{porous5,porous6,sto-por1} are studied much  less than the deterministic systems, while this kind of stochastic situations abound in real-world applications. In this paper, we consider stochastic parabolic partial differential models  in highly heterogeneous porous media with inherent multiple scales. When the coefficients of noise are relevant with the solution itself, these equations are called SPDEs driven by multiplicative noise, while the noise perturbation, independent of solution, corresponds to additive noise. SPDEs with multiplicative noise  are usually  more complicated  than SPDEs with additive noise. Moreover, nonlinearity of the noise perturbation  may result in a significant  challenge for simulating dynamical systems. To this end,  we develop  a model reduction to take care of  multiple scales and nonlinearity of SPDEs in porous media.

  Many multiscale methods have been developed  to effectively compute multiscale models in the past few years, such as  multiscale finite element methods (MsFEM) \cite{ms5,ms2}, multiscale finite volume method \cite{mfvm}, heterogeneous
multiscale methods \cite{hms}, mixed multiscale finite element methods  \cite{mixms2,mixms3,mixms1}, variational multiscale methods \cite{vms1,vms2,var1}, mortar multiscale methods \cite{mms1,mms2}, and generalized multiscale finite element method (GMsFEM) \cite{GMs2,GMs1}, localization  orthogonal decomposition \cite{mp14}, constraint energy minimizing generalized multiscale finite element method (CEM-GMsFEM)\cite{Eric2017CEM,Mengnan2019}.

In the paper, we use CEM-GMsFEM to solve multiscale SPDEs.  The construction  of CEM-GMsFE basis functions is usually divided into two stages. Firstly, we solve a local eigenproblem to get a set of auxiliary basis functions on each coarse block. Secondly, we solve a constraint energy minimization problem,  where the constraints are involved with auxiliary basis functions. When  CEM-GMsFEM is applied to nonlinear problems, the descretization system is still nonlinear and is defined in a fine grid space.  The evaluation of the nonlinear terms, as well as computation of the residual and the Jacobian matrix, are still all required on the fine grid, when we derive  the full discretization scheme by  a stable temporal implicit method.
Moreover, nonlinearity may lead to a slow convergence for iteration methods (e.g., Newton method) to solve the nonlinear problems in a high dimensional space.
Thus, the nonlinearity may bring a  significant computation challenge for solving  nonlinear multiscale SPDEs.

 We focus  on the nonlinear systems stemming from the spatial  discretization using  CEM-GMsFEM and seek an effective  model reduction method to improve the computation efficiency.  Proper orthogonal decomposition (POD) is one of popular  model reduction  methods  and  constructs a set of basis vectors with solution snapshots. These basis vectors span a low-dimensional reduced space which the full-order system is projected onto. But POD still needs the evaluation of nonlinear functions at all states, which  makes solving the reduced system as computationally expensive as solving the original high-dimensional problem. Thus POD is not appropriate to straightforwardly
   treat the nonlinear problems. To overcome the drawback of POD,  the discrete version of the empirical interpolation method \cite{eim}, DEIM \cite{deim-points,deim,o-deim} has been  proposed to effectively reduce the computation complexity of nonlinear problems.

 Typically, DEIM includes two phases. The first phase  is the offline phase where a reduced system is constructed using the solution of full-order system. The second phase  is the online phase in which we apply the precomputed reduced system to approximate the solution of new problem at hand. The offline phase computes the basis matrix and interpolation matrix by applying POD to a set of nonlinear snapshots, while the online phase uses the obtained  matrixes  to interpolate online nonlinear functions in a low dimension space.  The traditional DEIM  works well  for deterministic nonlinear  systems, but it may be not  sufficient for stochastic dynamical systems.  The solution of SPDE is a stochastic process,  which includes all  trajectories corresponding to different samples of noise. If we apply DEIM to the mean of snapshots  associated with  trajectories in the offline phase and adopt the reduced system for any new trajectory of solution, there may exist a large error because the precomputed reduced system does not contain any information from the new trajectory. Of course,  one can implement  DEIM individually  for each new trajectory. But it costs much online computation and is not suitable for the many-query applications such as  uncertainty quantification and the computation of statistics.
To this end, we propose a stochastic online DEIM to reduce the stochastic models. The proposed online DEIM also  includes two phases: offline phase and online phase.
In the offline phase, we utilize  a sort of  mean information of nonlinear functions   as the snapshots, and adopt DEIM to construct  the offline basis matrix and interpolation matrix.
In the online phase, the newly generated data ( evaluation of the nonlinear functions  along the new trajectory) is used to update the offline basis matrix in an increment manner.  The stochastic online DEIM
incorporates the precomputed model (mean of old trajectories)  and  the new trajectory information into the model reduction.
This approach can improve the fitness and prediction of the reduced model.
 A multiscale stochastic model reduction for SPDEs is constructed  by combining CEM-GMsFEM with stochastic online DEIM, and is referred to as  Online DEIM-MS in the paper.

This article is organized  as follows.  Section \ref{pre} is to introduce  some preliminary knowledge of SPDEs and weak formulations. In Section \ref{cem}, we
present   CEM-GMsFE  for SPDEs and give a priori error analysis for nonlinear SPDEs with multiplicative noise.
The stochastic online DEIM is addressed in  Section \ref{reduct}.
 In Section \ref{num}, a few numerical examples in porous media   are  presented to show the applicability and effectiveness of  Online DEIM-MS.
Some  conclusions are made finally.


\section{Preliminaries}\label{pre}

Let $D$ be a bounded spatial domain in $\mathbb{R}^{d}$ ($d=2$),  $\,T>0$  a fixed time, $H := L^2(D) $  a Hilbert space with norm $\|\cdot\|_H$. In this work, we consider the following nonlinear  stochastic  partial differential problem: find a $H$-valued stochastic process $u(t)$ to solve  the  SPDE with a multiplicative noise
\begin{equation}\label{ex-eq1}
\left\{
\begin{aligned}
du - \nabla \cdot \big(\kappa(x) \nabla u\big)dt&=f(u)dt+g(u)dW(t) \quad in\quad D \times (0,T] ,\\
u(x,t)&=0 \quad\qquad \qquad\quad\quad\quad\quad on\quad \partial D\times (0,T],\\
u(x,0)&=u_{0}(x) \quad\quad\quad\quad\quad\quad\quad\, in\quad D,
\end{aligned}
\right.
\end{equation}
or formally
\begin{equation}\label{ex-eq2}
\left\{
\begin{aligned}
u_t - \nabla \cdot \big(\kappa(x) \nabla u\big)&=f(u)+g(u)\xi(t) \quad in\quad D \times (0,T] ,\\
u(x,t)&=0 \quad\qquad \qquad\quad\quad on\quad \partial D\times (0,T],\\
u(x,0)&=u_{0}(x) \quad\quad\quad\quad\quad\, in\quad D,
\end{aligned}
\right.
\end{equation}
where $u_0(x)$ is the initial condition, $\kappa(x)$ is  a multiscale diffusion coefficient, $f$ and $g$ are $H$-valued nonlinear functionals, $W(t)$ is a $H$-valued $\mathcal{Q}$-Wiener process defined in a filtered probability space $(\Omega,\mathcal{F},\mathbb{P},{\left\{\mathcal{F}_t\right\}}_{t\geq 0})\cite{ok},$ and $\xi:=\frac{dW}{dt}$ is $H$-valued noise. {\color{black}The filtered probability space is assumed  to fulfill the usual conditions \cite{ito-integral0}. In the triple $(\Omega,\mathcal{F},\mathbb{P})$,  $\Omega$ is the sample space,  $\mathcal{F}$ is the $\sigma$-algebra of subset of $\Omega$, and a probability measure $\mathbb{P}$ on $(\Omega,\mathcal{F})$ is a function $\mathbb{P}:\mathcal{F}\rightarrow[0,1]$.} $\kappa(x)$ refers  to permeability and
is often heterogeneous in practical models.
In this paper, we follow the conventional notation, and suppress the dependence on space and write $u(t)$ for $u(t,x)$ and $W(t)$ for $W(t,x)$. 
{\color{black}Equation (\ref{ex-eq1}) can be viewed as a mathematical model
for the dynamics of diffusion-reaction flows driven by the stochastic perturbation $g(u)dW(t)$ in natural formations such as aquifers. In the modeling of solute transport problem in porous media, $u$ represents the solute concentration, while the dispersion of  the solute is primarily affected by the spatial heterogeneity or porosity \cite{sto-por1}.}

\par We assume that the covariance operator
$\mathcal{Q}: H \rightarrow H$ is non-negative, symmetric and of trace class, i.e., {\color{black} the series $Tr(\mathcal{Q}):=\sum_{i\in \mathbb{N}}(\mathcal{Q}\eta_i,\eta_i)_{H}$ converges, where $\{\eta_i\}_{i\in \mathbb{N}}$ is an arbitrary orthonormal basis of $H$, and $Tr(\mathcal{Q}) <\infty$}. Then the $\mathcal{Q}$-Wiener process \cite{QWiener,QWiener1} $W(t)$ can be represented as
\begin{equation}\label{KLE}
W(t)=\sum_{i\in \mathbb{N}} \sqrt{{\color{black}\mu_i}}e_i\beta_i(t), \quad t\in\left[0,T\right],
\end{equation}
where ${\color{black}\mu_i} \geq 0,\,e_i$ ($i\in \mathbb{N}$) are respectively the eigenvalues and the eigenfunctions of the covariance operator $\mathcal{Q}$ in $H$, and $\beta_i(t)$ are independent and identically distributed standard $\mathcal{F}_t$-Brownian motions.

For analysis, we review  some basic properties about the stochastic It$\hat{o}$ integral for integrable stochastic process with respect to $\mathcal{Q}$-Wiener process. {\color{black} Let  $L_2^0:=HS(\mathcal{Q}^{1/2}H,H)$
be the space of Hilbert-Schmidt operators from $\mathcal{Q}^{1/2}H$ to $H$, 
endowed with $\|\cdot\|_{L_2^0}$ norm, which is defined by
\[
\|\phi\|_{L_2^0}:=\|\phi\mathcal{Q}^{1/2}\|_{HS}=\left(\sum_{i=1}^{\infty}\|\phi\mathcal{Q}^{1/2}\eta_i\|_H^2   \right)^{1/2}, \quad \forall \phi \in L_2^0,
\]
where $\{\eta_i\}_{i=1}^{\infty}$ is an orthonormal basis in $H$. The space $L_2^0$ is separable.}  We refer to \cite{QWiener,ito-integral,ito-integral0} for the detailed introduction to these definitions.
 In this paper, we assume that the integrability of all $L_2^0$-valued stochastic processes is satisfied.
\begin{proposition}\cite{ito-integral0}\label{ito integral} (Martingale)
	For any $L_2^0$-valued integrable stochastic process $\phi$, the stochastic process
	\begin{equation*}
	M(t):=\int_{0}^{t}\phi dW(\tau)=0, \quad t\in[0,T],
	\end{equation*}
	is a continuous, square integrable martingale with
	\begin{equation}
	\boldsymbol{E}\left[M(t)\right]=0.
	\end{equation}
\end{proposition}
\begin{proposition}\cite{ito-integral0} (It$\hat{o}$ isometry)  For any $L_2^0$-valued integrable stochastic process $\phi$,
	\begin{equation}
	\boldsymbol{E}\left\|\int_{0}^{t}\phi dW(s)\right\|^2=\boldsymbol{E}\int_{0}^{t}\|\phi\|_{L_2^0}^2\,ds.
	\end{equation}
\end{proposition}

Suppose that the  coefficient $\kappa(x)$ is uniformly positive and bound in $\Omega$, i.e., there exists $0< \kappa_0 < \kappa_1<\infty$ such that $\kappa_0\leq\kappa(x)\leq \kappa_1$, but the ratio $\kappa_1/\kappa_0$ can be large.  Let operator  $\mathcal{A}:= -\nabla \cdot(\kappa(x)\nabla \cdot)$ and space
{\color{black}
\[
\mathcal{H}:=L^2\big(\Omega,L^2(D\times [0,T])\big)
\]
}
equipped with norm $||\cdot||_{\mathcal{H}}$.
We make the following assumptions to ensure  a unique mild solution of the equation (\ref{ex-eq1}) or (\ref{ex-eq2}).
\begin{itemize}
	\item[A.1] $f: H \rightarrow H  \, \text{ is globally Lipschitz continuous and of linear growth,} \ \  \text{i.e.,}$
	\begin{equation}
	\begin{aligned}
	\left\|f(u)-f(v)\right\|_H&\leq L_f\left\|u-v\right\|_{H},\\
	\left\|f(w)\right\|_H&\leq C_f(1+\|w\|_H), \quad \forall u,v,w \in H.
	\end{aligned}
	\end{equation}
	 Here $f \text{ is the  so-called Nemytskii operator defined for } H.$
	\item[A.2] $g: H \rightarrow L_0^2  \,\text{ is also globally Lipschitz continuous and of linear growth,} \ \ \text{i.e.,}$
	\begin{equation}
	\begin{aligned}
	\left\|g(u)-g(v)\right\|_{L^0_2}&\leq L_g\left\|u-v\right\|_{H},\\
	\left\|g(w)\right\|_{L^0_2}&\leq C_g(1+\left\|w\right\|_{H}), \quad \forall u,v,w \in H.
	\end{aligned}
	\end{equation}
	
	\item[A.3] $u_0 \in L^2(\Omega,L^2(\mathcal{D}(-\mathcal{A})^{1/2}))\,\text{ is measurable, }
	\text{i.e., }\, \boldsymbol{E}\left\|\nabla u_0\right\|^2< +\infty.$
	
\end{itemize}

\begin{theorem}\cite{mild-weak} Assume that A.1-A.3 hold. Then the equation (\ref{ex-eq1}) or (\ref{ex-eq2}) admits  unique mild solution in the form
	\begin{equation}\label{ex-eq3}
	u(t)=E(t)u_0+\int_{0}^{t}E(t-s)f(u(s))ds+\int_{0}^{t}E(t-s)g(u(s))dW(s),\, t\in [0,T],
	\end{equation}
	where $E(t):=e^{-t\mathcal{A}},t\ge 0$, is an analytic semigroup of the linear operator $-\mathcal{A}$, such that $\sup_{t\in[0,T]}\boldsymbol{E}\left\|u(t)\right\|_{H}^2 < +\infty$.
\end{theorem}


By  Theorem 9.15 in \cite{mild-weak},  the existence and uniqueness of the mild solution is equivalent to the existence and uniqueness of the  weak solution  under the assumptions A.1-A.3, i.e.,  there exists a unique weak solution $u(\cdot,t,\omega)\in H$ such that $\sup_{t\in[0,T]}\boldsymbol{E}\left\|u(t)\right\|_{H}^2 < +\infty$, and for $\forall v\in H_0^1, \, t\in [0,T]$, we have
\begin{equation}\label{weak_sol}
(u(t),v)=(u_0,v)-\int_{0}^{t}(\kappa(x)\nabla u(s),\nabla v)ds+\int_{0}^{t}(f(u(s)),v)ds+\int_{0}^{t}(g(u(s))dW(s),v),
\end{equation}
where
\[
\int_{0}^{t}(g(u(s))dW(s),v):=\sum_{i=1}^{\infty}\int_0^t(g(u(s))\sqrt{{\color{black}\mu_i}}e_i,v)d\beta_i(s).
\]
Thus we get the variational form of (\ref{ex-eq2}),
\begin{equation}\label{eq-weak}
\left\{
\begin{aligned}
(u_t,v) + a(u,v)& = (f(u),v)+(g(u)\xi(t),v) \quad \forall v\in H_0^1, \, t> 0,\\
\big(u(\cdot,0), v\big)&=(u_{0},v) \quad \forall v\in H_0^1,
\end{aligned}
\right.
\end{equation}
where
\[
a(u,v) = (\kappa \nabla u,\nabla v)=\int_{D} \kappa \nabla u \cdot \nabla v, \quad (f,v)=\int_{D}fv.
\]
The first aim of this paper is to carry out a CEM-MsFE convergence analysis  for the equation  (\ref{ex-eq1}). To this end, we construct a  multiscale finite element  space $V_{ms} \subset H_0^1$. The CEM-GMsFE
solution  $u(\cdot,t)\in V_{ms}$ solves
\begin{equation}\label{eq:scheme}
\left\{
\begin{aligned}
\big((u_{ms})_t,v\big) + a(u_{ms},v)& = (f(u_{ms}),v)+(g(u_{ms})\xi(t),v) \quad \forall v\in V_{ms},\, t> 0,\\
\big(u_{ms}(\cdot,0), v\big)&=(u_{0,ms},v) \quad \forall v\in V_{ms},
\end{aligned}
\right.
\end{equation}
where $u_{0,ms}\in V_{ms}$ is the $L^2$ projection of $u_0$ onto $V_{ms}$.

\section{CEM-GMsFE method for SPDE} \label{cem}
In this section, we first present the construction for multiscale basis functions. Then a priori error analysis is carried out for the SPDE (\ref{ex-eq1})
in CEM-GMsFEM.

\subsection{Construction of multiscale basis functions}

This subsection will follow the  works \cite{GMs2,Eric2017CEM,GMs1,Mengnan2019}, and focus on the construction of the CEM-GMsFEM basis functions for the equation (\ref{ex-eq1}). The multiscale finite element  space is constructed by solving  constrained energy minimization problems.

Let $\mathcal{T}_{h}$ be a coarse gird partition of the space domain $\Omega$ with $N_c$ nodes, $N$ elements, and $h$  the size of a generic  coarse element. Figure \ref{coarsegrid} shows the fine gird, the coarse grid $K_i$ and oversampling domain {\color{black}$K_i^{m}$} by enlarging  coarse layers ($m=2$).
\begin{figure}[hbtp]
	\centering
	\includegraphics[width=3.8in, height=3in]{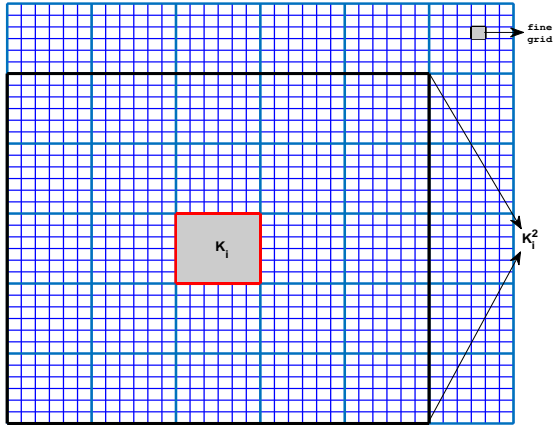}
	\caption{The fine gird, coarse grid $K_i$, oversampling domain $K_i^{2}$}\label{coarsegrid}
\end{figure}

\par For each coarse block $K_i\in \mathcal{T}_{h}$, we denote $V(K_i)\,:=\, H^1(K_i)$, and consider a set of auxiliary basis functions $\{{\phi}_j^{(i)}\}_{j=1}^{L_i}\in V(K_i) $ by solving the spectral problem
\begin{equation}
a_i({\phi}_j^{(i)},v)={\lambda}_j^{(i)}s_i({\phi}_j^{(i)},v),\quad \forall v\in V(K_i),
\end{equation}
 where $s_i(\phi_j^{(i)},v)=\int_{K_i} \tilde{\kappa}uv$, and $\tilde{\kappa}=\kappa \sum_i |\Delta {\chi}_i|^2$,  $\{{\chi}_i\}$ is a set of partition of unity functions for the coarse partition. We collect the first $L_i$ eigenfunctions corresponding to the first $L_i$ smallest eigenvalues for the local auxiliary multiscale space $V_{aux}^{(i)}=\text{span} \{\phi_j^{(i)}|1\leq j\leq L_i\}$ and define the global auxiliary space $V_{aux}$ by the  direct sum of these local spaces,
\[
V_{aux}=\bigoplus_{i=1}^N V_{aux}^{(i)}.
\]
Let the oversampling domain $K_{i}^m\subset \Omega $ be  the extension  of $K_i$ with $m$ coarse grid layers (as in Figure \ref{coarsegrid}). The  CEM-MsFE basis functions  $\varphi_{j,ms}^{(i)}\in H_0^1(K_i^m)$  are constructed through solving the local problems
\begin{equation}\label{localbasis}
\varphi_{j,ms}^{(i)}= \arg\min\left\{\,a(\varphi,\varphi)\,|\,\varphi \in H_0^1(K_i^m),\,\varphi \text{ is } \phi_j^{(i)}\text{-orthogonal}\right\},
\end{equation}
where $\varphi \in V$ is $\phi_j^{(i)}\text{-orthogonal}$ when
\begin{equation*}
s(\varphi, \phi_j^{(i)})=1 \quad \text{ and } \quad s(\varphi, \phi_{j'}^{(i')})=0, \quad \text{if } j'\neq j \text{ or } i'\neq i,
\end{equation*}
with $s(u,v)=\sum_{i=1}^Ns_i(u,v).$
Thus, we define the multiscale finite element space for CEM-GMsFEM by
\[
V_{ms}=\text{ span }\left\{ \varphi_{j,ms}^{(i)}\,|1\leq j\leq L_i,\, 1\leq i\leq N\right\}.
\]
{\color{black}
	\begin{remark}
		 Lemma 5 in \cite{Eric2017CEM} implies that  more  layers $m$ of the oversampling domain $K_i^m$ can  improve the  approximation accuracy  of the multiscale method. But the multiscale basis functions decay to zero on the oversampling domain as more layers are utilized \cite{Eric2017CEM,Mengnan2019}.  Thus, we may use a small number of layers
for the oversampling domain to balance the approximation accuracy and computation efficiency.
	\end{remark}
}

\subsection{A priori error analysis}
In this subsection, we approximate the solution $u(x,t)$ of the problem (\ref{eq-weak}) by a function $u_{ms}\in V_{ms}$ for each $t>0$.
First of all, we provide the error estimates of CEM-GMsFEM for the SPDE with multiplicative noise. To this end,  we need  some preliminary lemmas.

\begin{lemma}\label{lem14}
	Let $u$ be the solution of the stochastic parabolic  equation (\ref{ex-eq1}). Then
	\begin{equation}\label{ut}
	\boldsymbol{E}\|u_{t}\|^{2}_{L^2(D\times [0,T])}\leq C\bigg(\boldsymbol{E}\|u_{0}\|_a^{2}+\|f(u)\|^{2}_{\mathcal{H}} \bigg),
	\end{equation}
	{\color{black}where $\|u\|_{a}^{2}:=\int_{D}\kappa|\nabla u|^{2}$ is the energy norm.}
	\begin{proof}
		By (\ref{eq-weak}), we have
		\[
		(u_{t},u_{t}) + a(u,u_t) =(f(u),u_{t})+(g(u)\xi(t),u_t),
		\]
		which implies
		\[
		\|u_t\|^2_{H} + \frac{1}{2} \frac{d}{dt} \|u\|_a^2 = (f(u),u_t)+(g(u)\xi(t),u_t).
		\]
		Integrating over the interval $[0,t],\, 0\leq t \leq T$ and taking the expectation of both sides, we get
		\begin{align*}
		&\boldsymbol{E}\int_0^t \|u_t(x,s)\|^2_{H} ds+ \frac{1}{2} \boldsymbol{E}\|u(x,t)\|_a^2 \\
		&= \boldsymbol{E}\int_0^t (f(u),u_t)ds +\boldsymbol{E}\int_0^t (g(u)\xi(s),u_t)ds+ \frac{1}{2} \boldsymbol{E}\|u_0\|_a^2\\
		&=\boldsymbol{E}\int_0^t (f(u),u_t)ds +\boldsymbol{E}\int_0^t (g(u)dW(s),u_t)+ \frac{1}{2} \boldsymbol{E}\|u_0\|_a^2\\
		&\leq \boldsymbol{E}\int_0^t \frac{1}{2}\left(\|f(u)\|^2_{H}+\|u_t\|^2_{H}\right)ds +\frac{1}{2} \boldsymbol{E}\|u_0\|_a^2 ,
		\end{align*}
		where  $\boldsymbol{E}\int_0^t (g(u(s))dW(s),u_t)=0$ has been  used in the third step by Prop.\ref{ito integral}. This completes the proof.	
\end{proof}
	
\end{lemma}

\begin{remark}
	Since $f(u)$ is of linear growth and a Nemytskii operator in $H$ and the weak solution of (\ref{ex-eq2}) $u \in L^{\infty}([0,T],L^2(D\times \Omega))$, i.e., $\left\|f(u)\right\|_{H}\leq C_f(1+\left\|u\right\|_{H}), \forall u\in H$ and $\sup_{t\in[0,T]}\boldsymbol{E}\left\|u(t)\right\|_{H}^2 <+\infty$, then further we have the following estimation by Lemma \ref{lem14},
	\begin{align*}
		\boldsymbol{E}\|u_{t}\|^{2}_{L^2(D\times [0,T])}&\leq C\bigg(\boldsymbol{E}\|u_{0}\|_a^{2}+\|f(u)\|^{2}_{\mathcal{H}} \bigg)\\
	&\leq C\bigg(\boldsymbol{E}\|u_{0}\|_a^{2}+2C_f(1+\|u\|^2_{\mathcal{H}}) \bigg)\\
	&\leq C\bigg(\boldsymbol{E}\|u_{0}\|_a^{2}+2C_f\left(1+\|u\|^2_{L^{\infty}\left([0,T],L^2(D\times \Omega)\right)}\right) \bigg)\\
	&< +\infty.
	\end{align*}	
\end{remark}

Let $\widehat{u}\in V_{ms}$ be the elliptic projection of the function $u\in V$, i.e., $\widehat{u}$ satisfies
\[
a(\widehat{u},v)=a(u,v),  \, \forall v\in V_{ms}.
\]
	From  Lemma \ref{lem14}, we immediately have the following estimates of $\widehat{u}$ and $u_{ms}$, respectively,
	\begin{equation}\label{ut^}
	\begin{split}
	\boldsymbol{E}\|\widehat{u}_{t}\|^{2}_{L^2(D\times [0,T])} &\: \leq C\bigg(\boldsymbol{E}\|\widehat{u}_{0}\|_a^{2}+\|f(u)\|^{2}_{\mathcal{H}} \bigg),\\
	\boldsymbol{E}\|(u_{ms})_{t}\|^{2}_{L^2(D\times [0,T])} &\: \leq C\bigg(\boldsymbol{E}\|u_{0,ms}\|_a^{2}+\|f(u)\|^{2}_{\mathcal{H}} \bigg).
	\end{split}
	\end{equation}
By Lemma $4.3$ in \cite{Mengnan2019}, we have the following estimate between the solution of stochastic parabolic equation and its corresponding elliptic projection.
\begin{lemma}\label{lem1}
	Let $u$ be the solution of (\ref{ex-eq1}). For each $t>0$, any fixed realization $\xi(t)$,  the elliptic projection $\widehat{u}(t)\in V_{ms}$ and satisfies
	\begin{equation}\label{ms-ell}
	a\big(u(t)-\widehat{u}(t),v\big) =0, \quad \forall v\in V_{ms}.
	\end{equation}
	Then, for any $t>0$,
	\begin{eqnarray*}
	&\| (u-\widehat{u})(t)\|_{a} \leq Ch\Lambda^{-\frac{1}{2}} \|\kappa^{-\frac{1}{2}}(f(u)+g(u)\xi(t)-u_{t})(t)\|_{H},\\
   &\|(u-\widehat{u})(t)\|_{H}\leq Ch^{2}\Lambda^{-1} \kappa_0^{-\frac{1}{2}} \|\kappa^{-\frac{1}{2}}(f(u)+g(u)\xi(t)-u_{t})(t)\|_{H},
	\end{eqnarray*}
	where $C$ is a constant independent of $\kappa$ and the mesh size $h${\color{black}, and $\Lambda=\mathop{\min}\limits_{1\leq i\leq N}\lambda_{L_i+1}^{(i)}$  is the smallest eigenvalue that we discard when constructing the local auxiliary multiscale space.}
\end{lemma}
\begin{remark}\label{rem}
	By Lemma \ref{lem1} and Lemma \ref{lem14}, assuming $g(u)\xi(t)\in \mathcal{H}$, we further have the following  estimations,
	\begin{equation}\label{ms-a-err}
	\begin{split}
	\boldsymbol{E}\int_0^t\| (u-\widehat{u})(s)\|^2_{a}ds  \leq Ch^2\Lambda^{-1}\kappa_0^{-1 }\left(\|f(u)\|^2_{\mathcal{H}} +\|g(u)\xi(t)\|^2_{\mathcal{H}}+\boldsymbol{E}\|u_0\|_a^2\right),
	\end{split}	
	\end{equation}
	\begin{equation}\label{ms-l2-err}
	\boldsymbol{E}\int_0^t\| (u-\widehat{u})(s)\|^2_{H}ds\leq Ch^4\Lambda^{-2}\kappa_0^{-2 }\left(\|f(u)\|^2_{\mathcal{H}} +\|g(u)\xi(t)\|^2_{\mathcal{H}}+\boldsymbol{E}\|u_0\|_a^2\right).
	\end{equation}
\end{remark}

In the following theorem, we  prove a strong convergence of the semidiscrete solution of  (\ref{eq:scheme}).
\begin{theorem}\label{thm2}
	Let $u$ and $u_{ms}$ be the solution of (\ref{eq-weak}) and (\ref{eq:scheme}),  respectively. Then for any fixed $t \in [0,T]$,
	\begin{equation}
	\begin{split}
	&\:  \boldsymbol{E}\| (u - u_{ms})(t) \|_{H}^{2} + \boldsymbol{E}\int_{0}^{t}\| u - u_{ms} \|_a^{2}ds\\
	&\:\leq Ch^{2}\Lambda^{-1} \kappa_0^{-1} \left(\|f(u)\|^2_{\mathcal{H}} +\|g(u)\xi(x,t)\|^2_{\mathcal{H}}+\boldsymbol{E}\|u_0\|_a^2+\boldsymbol{E}\|u_{0,ms}\|_a^2\right)\\
	\quad  &\: + \boldsymbol{E}\|(u-u_{ms})(0)\|_{H}^2,
	\end{split}
	\end{equation}
	where $C$ is a constant independent of $\kappa$ and the mesh size $h$.
\end{theorem}
\begin{proof}
	By  (\ref{eq-weak}) and (\ref{eq:scheme}), we have
	\begin{align}
	\label{weaku}(u_t,v) + a(u,v) & = \big(f(u),v\big)+\big(g(u)\xi(t),v\big) ,
	\\
	\label{weakums}\big((u_{ms})_t,v\big) + a(u_{ms},v) & = \big(f(u_{ms}),v\big)+\big(g(u_{ms})\xi(t),v\big)
	\end{align}
	for all $v\in V_{ms}$. Thus, by subtracting (\ref{weakums}) from (\ref{weaku}), we get, for $\forall v\,\in V_{ms}$
	\begin{equation}\label{diff}
	\big((u-u_{ms})_t,v\big) + a(u-u_{ms},v)  = \big(f(u)-f(u_{ms}),v\big)+\big(g(u)\xi-g(u_{ms})\xi,v\big).
	\end{equation}
	The elliptic projection of $u$, $\widehat{u}\in V_{ms}$, satisfies (\ref{ms-a-err}) and (\ref{ms-l2-err}).  In particular, by taking $v=\widehat{u}-u_{ms}\in V_{ms} $ in (\ref{diff}),
	we have the following error equation,
	\begin{equation}\label{adiff}
	\begin{split}
	&\big((u-u_{ms})_t,\widehat{u}-u_{ms}\big) +\: a(u-u_{ms},\widehat{u}-u_{ms})\\
 =&\:\big(f(u)-f(u_{ms}),\widehat{u}-u_{ms}\big)+\Big(\big(g(u)-g(u_{ms})\big)\xi,\widehat{u}-u_{ms}\Big) .	
	\end{split}
	\end{equation}
	Due to the bilinearity of $a(\cdot,\cdot)$, we can rewrite (\ref{adiff}) as
	\[
	\begin{split}
	&\: \big((u-u_{ms})_t,u-u_{ms}\big) + a(u-u_{ms},u-u_{ms})\\
	&\:+ \big(f(u_{ms})-f(u),u-u_{ms}\big)+\Big(\big(g(u_{ms})-g(u)\big)\xi,u-u_{ms}\Big)\\=
	&\:\big((u-u_{ms})_t,u-\widehat{u}\big) + a(u-u_{ms},u-\widehat{u})\\
	&\:+ \big(f(u_{ms})-f(u),u-\widehat{u}\big)+\Big(\big(g(u_{ms})-g(u)\big)\xi,u-\widehat{u}\Big).
	\end{split}
	\]
	The Lipschitz assumptions of $f$, Cauchy-Schwarz inequality and Young's inequality give
	\[
	\begin{split}
	&\:\frac{1}{2} \frac{d}{dt} \| u - u_{ms} \|^2 + \| u - u_{ms}\|_{a}^2 -L_f\|u-u_{ms}\|^2+\Big(\big(g(u_{ms})-g(u)\big)\xi,u-u_{ms}\Big)
	\\\leq &\:\big((u-u_{ms})_t,u-\widehat{u}\big) + a(u-u_{ms},u-\widehat{u})+\frac{1}{2}L_f\|u-u_{ms}\|^2+\frac{1}{2}\|u-\widehat{u}\|^2\\&\:+\Big(\big(g(u_{ms})-g(u)\big)\xi,u-\widehat{u}\Big).
	\end{split}
	\]
	After taking  integration over $[0,t]$ for $\forall t, \, 0\leq t\leq T$, and taking  expectation with respect to random process, and using  the Cauchy-Schwarz and Young's inequalities to the right-hand  side, it  follows that
	\begin{equation}\label{IE}
	\begin{split}
	&\boldsymbol{E}\| (u - u_{ms}) (t)\|_{H}^2 + \boldsymbol{E}\int_{0}^{t}\| u - u_{ms} \|_{a}^2ds \\
&\leq  2\boldsymbol{E}\int_0^t(\|u_t\|_{H}+\|(u_{ms})_t\|_{H})\| u - \widehat{u}\|_{H} ds
	 + \boldsymbol{E}\int_0^t\| u-\widehat{u}\|_{a}^2ds+\boldsymbol{E}\int_0^t\|u-\widehat{u}\|_{H}^2ds\\
	&+3L_f\boldsymbol{E}\int_0^t\|u-u_{ms}\|_{H}^2ds
	+\boldsymbol{E}\|(u-u_{ms})(0)\|_{H}^2,
	\end{split}
	\end{equation}
	where we used the following result by Prop.\ref{ito integral},
\begin{eqnarray*}
 &\boldsymbol{E}\int_0^t\Big(\big(g(u_{ms})-g(u)\big)\xi(s),u-u_{ms}\Big)ds=\boldsymbol{E}\int_0^t\Big(\big(g(u_{ms})-g(u)\big)dW(s),u-u_{ms}\Big) ,\\ &\boldsymbol{E}\int_0^t\Big(\big(g(u_{ms})-g(u)\big)\xi(s),u-\widehat{u}\Big)ds=\boldsymbol{E}\int_0^t\Big(\big(g(u_{ms})-g(u)\big)dW(s),u-\widehat{u}\Big).
 \end{eqnarray*}
For  the first term on the right-hand  side of (\ref{IE}),  H$\ddot{o}$lder inequality implies
	\begin{equation}\label{holder}
	\begin{split}
	&\boldsymbol{E}\int_0^t\big(\|u_t\|_{H} \:+\|(u_{ms})_t\|_{H}\big)\| u - \widehat{u}\|_{H} ds   \\
	&\: \leq \boldsymbol{E}\left\{\left(\int_0^t\big(\|u_t\|_{H}+\|(u_{ms})_t\|_{H}\big)^2 ds\right)^{\frac{1}{2}}\left(\int_{0}^t\| u - \widehat{u}\|_{H}^2 ds\right)^{\frac{1}{2}}\right\}.
	\end{split}
	\end{equation}
	For presentation convenience,  we define
	\begin{equation}\label{denote}
	\begin{split}
	I := &\:2\boldsymbol{E}\left\{\left(\int_0^t\big(\|u_t\|_{H}+\|(u_{ms})_t\|_{H}\big)^2 ds\right)^{\frac{1}{2}}\left(\int_{0}^t\| u - \widehat{u}\|_{H}^2 ds\right)^{\frac{1}{2}}\right\} \\
	&\:+ \boldsymbol{E}\int_0^t\| u-\widehat{u}\|_{a}^2ds+\boldsymbol{E}\int_0^t\|u-\widehat{u}\|_{H}^2ds +\boldsymbol{E}\|(u-u_{ms})(0)\|_{H}^2 .
	\end{split}
	\end{equation}
	Hence, from (\ref{IE}), (\ref{holder}) and (\ref{denote}), we get
	\[
	\begin{split}
	\boldsymbol{E}\| (u - u_{ms}) (t)\|_{H}^2 +&\: \boldsymbol{E}\int_{0}^{t}\| u - u_{ms} \|_{a}^2ds \leq I +3L_f\boldsymbol{E}\int_0^t\|u-u_{ms}\|_{H}^2ds\\
	&\: \leq I+3L_f\boldsymbol{E}\int_0^t\left(\|u-u_{ms}\|_{H}^2+\int_{0}^{t}\| u - u_{ms} \|_{a}^2d\tau\right)ds.
	\end{split}
	\]
	The Gronwall inequality concludes that
	\begin{equation}\label{gronwall}
	\boldsymbol{E}\| (u - u_{ms}) (t)\|_{H}^2  + \boldsymbol{E}\int_{0}^{t}\| u - u_{ms} \|_{a}^2ds \leq \mathcal{C} I.
	\end{equation}
	Finally, combining Lemma \ref{lem14}, Corollary \ref{ut^}, Lemma \ref{lem1}, Remark \ref{rem} and (\ref{gronwall}), we complete the proof.
\end{proof}

In particular,  for stochastic parabolic PDEs with additive noise, which is the degeneration of (\ref{ex-eq1}), i.e.,
\begin{equation}\label{additive}
\left\{
\begin{aligned}
du - \nabla \cdot \big(\kappa(x) \nabla u\big)dt&=f(u)dt+g(x,t)dW(t) \quad in\quad D \times (0,T] ,\\
u(x,t)&=0 \quad on\quad \partial D\times (0,T],\\
u(x,0)&=u_{0}(x) \quad in\quad D.
\end{aligned}
\right.
\end{equation}
We similarly have the variational form of (\ref{additive})
\begin{equation}\label{add-weak}
\left\{
\begin{aligned}
(u_t,v) + a(u,v)& = \big(f(u),v\big)+\big(g(x,t)\xi(t),v\big) \quad \forall v\in H_0^1, \, t> 0,\\
\big(u(\cdot,0), v\big)&=(u_{0},v) \quad \forall v\in H_0^1,
\end{aligned}
\right.
\end{equation}
Then CEM-GMsFEM solution $u(\cdot,t)\in V_{ms}$ of (\ref{add-weak}) solves
\begin{equation}\label{add:scheme}
\left\{
\begin{aligned}
\big((u_{ms})_t,v\big) + a(u_{ms},v)& = \big(f(u_{ms}),v\big)+\big(g(x,t)\xi(t),v\big) \quad \forall v\in V_{ms},\, t> 0,\\
\big(u_{ms}(\cdot,0), v\big)&=(u_{0,ms},v) \quad \forall v\in V_{ms}
\end{aligned}
\right.
\end{equation}

We have the following error estimate of CEM-GMsFEM  for   the SPDE with additive noise, which generalizes the result in \cite{zhang21}.
\begin{theorem}\label{thm3}
	Let $u$ and $u_{ms}$ be the solution of (\ref{add-weak}) and (\ref{add:scheme}),  respectively. Then for any fixed $t \in [0,T]$,
	\begin{equation}
	\begin{split}
	&\:\boldsymbol{E}\| (u - u_{ms})(t) \|_{H}^{2} + \boldsymbol{E}\int_{0}^{t}\| u - u_{ms} \|_a^{2}ds\\
	&\:\leq Ch^{2}\Lambda^{-1} \kappa_0^{-1} \left(\|f(u)\|^2_{\mathcal{H}} +\boldsymbol{E}\|u_0\|_a^2+\boldsymbol{E}\|u_{0,ms}\|_a^2\right) + \boldsymbol{E}\|(u-u_{ms})(0)\|_{H}^2,
	\end{split}
	\end{equation}
	where $C$ is a constant independent of $\kappa$ and the mesh size $h$.
\end{theorem}
\begin{proof}
	From  (\ref{add-weak}) and (\ref{add:scheme}), for $\forall v\,\in V_{ms}$, we have
	\begin{equation}\label{add-diff}
	\big((u-u_{ms})_t,v\big) + a(u-u_{ms},v)  = \big(f(u)-f(u_{ms}),v\big).
	\end{equation}
The  noise term disappears because diffusion $g(x,t)$ is independent of the state $u$.
	By following  the proof of Theorem \ref{thm2}, the proof is done.
\end{proof}

\section{Model reduction for nonlinear terms}\label{reduct}
Let  $\{\varphi_i(x)\}_{i=1}^{N_r}$ be the set of multiscale basis functions in CEM-GMsFEM.   Then we have an  approximation,
\[
u\approx u_{ms}= \sum_{i=1}^{N_r}u^i_{ms}(t) \varphi_i(x).
\]
After  we apply  CEM-GMsFEM to SPDE(\ref{ex-eq1}), CEM-GMsFEM yields a discrete dynamical system of nonlinear ODE:
\begin{equation}\label{ms}
\left\{\begin{aligned}
\mathbf{M}\frac{d \mathbf{u_{ms}}}{dt}+A\mathbf{u_{ms}}(t)&=\mathbf{F}(\mathbf{u_{ms}}(t))\\
\mathbf{u_{ms}^0}&= \mathbf{u_{ms}}(0),
\end{aligned}
\right.
\end{equation}
where $\mathbf{u_{ms}}=[u^1_{ms},u^2_{ms},\cdots,u^{N_r}_{ms}]^T\in \mathbb{R}^{N_r}$ and
\begin{equation*}
\begin{aligned}
\mathbf{M}_{ij}&=\int_{D}\varphi_i(x)\varphi_j(x),\\
A_{ij}&=\int_D \kappa(x)\nabla\varphi_i(x)\nabla\varphi_j(x),\\
{\mathbf{F}(\mathbf{u_{ms}}(t))}_{i}&\approx\int_D \bigg(f\big(\sum_{l=1}^{N_r}u^l_{ms}\varphi_l(x)\big)+g\big(\sum_{l=1}^{N_r}u^l_{ms}\varphi_l(x)\big)\xi(t)\bigg)\varphi_i(x).
\end{aligned}
\end{equation*}
Notice that $\mathbf{M},\,A \in \mathbb{R}^{N_r\times N_r}$, and $\mathbf{F}: \mathbb{R}^{N_r} \rightarrow \mathbb{R}^{N_r} $ is a nonlinear function. Although $N_r$, number of multiscale basis functions, is usually much smaller than number of fine-grid FEM basis functions,  computation  of the above system (\ref{ms}) is still challenging because the high-dimensionality is still in the nonlinear term and we need to compute many trajectories for
 statistical estimation of the solution.  Thus it is desirable to significantly reduce the nonlinear dynamic system.

In this section, we first briefly review two important model reduction methods:  POD and DEIM.  Finally we introduce a  stochastic online DEIM, which can substantially improve the computation efficiency of the nonlinear stochastic  system (\ref{ms}). For simplicity of presentation, we rewrite the system (\ref{ms}) by
\begin{equation}\label{nonlinear}
\frac{d \mathbf{y}}{dt}=A\mathbf{y}(t)+\mathbf{f}(\mathbf{y}(t)).
\end{equation}
 Here  we have abused  the notation for linear coefficient matrix $A$ and nonlinear function $\mathbf{f}$ corresponding to $\mathbf{M}^{-1}A$ and $\mathbf{M}^{-1}\mathbf{F}$ respectively in (\ref{ms}), and $\mathbf{y}$ represent the unknown solution that we are interested in.

\subsection{Proper orthogonal decomposition}

In the POD approach, we simulate the system by a reduced basis. Let  $\{\mathbf{y}(t_i)\}_{i=1}^M$ be the solution vectors at $M$ time levels, and
 \[
 Y=[\mathbf{y}(t_1),\mathbf{y}(t_2),\cdots,\mathbf{y}(t_M)] \in \mathbb{R}^{n\times M}
  \]
  be the snapshot matrix. POD constructs an orthogonal basis $V\in \mathbb{R}^{n\times r}$, in which the basis vectors are the first $r\leq \text{rank}(Y)$  left singular vectors of $Y$ corresponding to the first $r$ largest singular values. The basis is also a solution of the following minimization problem
\[
\min_{V\in \mathbb{R}^{n\times r}, V^{T}V=I_r} \|(I_n-VV^{T})Y\|_F^2.
\]
Thus  POD-Galerkin method gives a reduced system of (\ref{nonlinear})
\begin{equation}\label{reduced}
\frac{d \hat{\mathbf{y}}}{dt}=V^{T}AV\hat{\mathbf{y}}(t)+V^{T}\mathbf{f}(V\hat{\mathbf{y}}(t)),
\end{equation}
and the solution $\mathbf{y}(t)$ of (\ref{nonlinear}) is approximated by $\mathbf{y}(t)\approx V\hat{\mathbf{y}}(t)$. However, solving the reduced system (\ref{reduced}) instead of (\ref{nonlinear}), we still need evaluations of the nonlinear term $\mathbf{f}$ at all $n$ components,  the simulation of  (\ref{reduced}) is as expensive as solving the original system.
DEIM can be used  to overcome this difficulty.

\subsection{Discrete empirical interpolation method}\label{deim}
DEIM provides an approximation of the nonlinear function $\mathbf{f}$ by a projection (interpolation), which maps a high $n-$dim space to a low $m-$dim space ($m\ll n$)  spanned by the basis $\mathbf{U}=[\mathbf{u}_1,\mathbf{u}_2,\cdots,\mathbf{u}_m]\in \mathbb{R}^{n\times m}$.  DEIM gives
\begin{equation}\label{f=uc}
\mathbf{f}\approx \mathbf{U}\mathbf{c},
\end{equation}
where $\mathbf{c}$ is an unknown vector. The projection basis $\mathbf{U}$ is obtained by applying POD to an  appropriate set of the nonliear snapshots
\[
\{\mathbf{f}(\mathbf{y}(t_1)),\cdots,\mathbf{f}(\mathbf{y}(t_M))\} \subset \mathbb{R}^{n}.
\]
The system (\ref{f=uc}) is overdetermined. In order to solve it, we multiplies the equation by the transpose of the matrix $P\in \mathbb{R}^{n\times m}$,
\[
P^T\mathbf{f}=P^T\mathbf{U}\mathbf{c}.
\]
Thus, if $P^T\mathbf{U}$ is non-singular, then we can determine $\mathbf{c}=(P^T\mathbf{U})^{-1}P^T\mathbf{f}$. Therefore, we have an  approximation of nonlinear term,
\begin{equation}\label{nl-app}
\mathbf{f}\approx \mathbf{U}(P^T\mathbf{U})^{-1}P^T\mathbf{f}.
\end{equation}
In particular, DEIM chooses $m$ pairwise different interpolation points $p_1,\cdots,p_m \in \{1,\cdots,n\}$ and $P=[e_{p_1},e_{p_2},\cdots,e_{p_m}],$ where $e_{p_i}=[0,\cdots,0,\cdots,1,0,\cdots,0]^T\in \mathbb{R}^n$ is the $p_i$-th unit vector. The interpolation points are chosen by a greedy method \cite{deim-points},    which is specified by the basis matrix.  The interpolation points matrix $P$ and the DEIM basis $\mathbf{U}$ are selected such that $P^T\mathbf{U}$ has full rank.
The approximation (\ref{nl-app}) implies that the nonlinear term $\mathbf{f}$ can be only evaluated at $m$ entries specified by $P^T$ instead of all $n$ components.  Thus DEIM gives a reduced system,
\[
\frac{d \hat{\mathbf{y}}}{dt}=V^{T}AV\hat{\mathbf{y}}(t)+V^{T}\mathbf{U}(P^T\mathbf{U})^{-1}P^T\mathbf{f}(V\hat{\mathbf{y}}(t)).
\]

\subsection{Stochastic online DEIM}\label{onlinedeim}
For stochastic systems, we may have the nonlinear snapshots along several different  trajectories $\{\mathbf{y}^1,\mathbf{y}^2,\mathbf{y}^3, \cdots,\mathbf{y}^k\}$ at times $\{t_1,t_2,\\\cdots,t_M\}$:
\begin{equation*}
\begin{aligned}
\{\mathbf{f}(\mathbf{y}^1(t_1)),\cdots,\mathbf{f}(\mathbf{y}^1(t_M))\} \subset \mathbb{R}^{n},\\
\{\mathbf{f}(\mathbf{y}^2(t_1)),\cdots,\mathbf{f}(\mathbf{y}^2(t_M))\}\subset \mathbb{R}^{n},\\
\{\mathbf{f}(\mathbf{y}^3(t_1)),\cdots,\mathbf{f}(\mathbf{y}^3(t_M))\}\subset \mathbb{R}^{n},	\\
\vdots\quad\quad \qquad\quad\,\,\quad \qquad \\
\{\mathbf{f}(\mathbf{y}^k(t_1)),\cdots,\mathbf{f}(\mathbf{y}^k(t_M))\}\subset \mathbb{R}^{n}.
\end{aligned}
\end{equation*}	
 Although we can apply DEIM to each trajectory, the computation complexity will be  proportional to the number of trajectories.
  To take account of this issue, we apply DEIM to mean of a few  trajectories available. Let the mean snapshots
  \[\{\mathbf{f}(\bar{\mathbf{y}}(t_1)),\cdots,\mathbf{f}(\bar{\mathbf{y}}(t_M))\} \subset \mathbb{R}^{n},
  \]
where $\mathbf{f}(\bar{\mathbf{y}}(t_i))= \mathbf{f}\big(\frac{1}{k}\sum_{j=1}^k \mathbf{y}^j(t_i)\big)$ $i=1,2,\cdots,M$. Thus we get the approximation
\[
\mathbf{f}(\bar{\mathbf{y}}(t))\approx \bar{\mathbf{U}}(\bar{P}^T\bar{\mathbf{U}})^{-1}\bar{P}^T\mathbf{f}(\bar{\mathbf{y}}(t)),
\]
where the DEIM basis $\bar{\mathbf{U}}$ is obtained by the SVD of snapshot matrix $[\mathbf{f}(\bar{\mathbf{y}}(t_1)),\cdots,\\\mathbf{f}(\bar{\mathbf{y}}(t_M))]$, and $\bar{P}$ is the interpolation point matrix.

 However, when a new trajectory is realised, i.e.,  $ \{\mathbf{f}(\tilde{\mathbf{y}}(t_1)),\cdots,\mathbf{f}(\tilde{\mathbf{y}}(t_M))\} \subset \mathbb{R}^{n}$ is available, the reduced model obtained by dealing with the mean of all previous trajectories  may  not capture the new information from the new trajectory. In order to take the  new information, we introduce a  stochastic online adaptive DEIM, which  updates the orthogonal basis matrix online with respect to any new trajectory.

 Based on the reduced model  $(\bar{\mathbf{U}},\bar{P})$   obtained already in the offline phase, we provide the online updated basis $\tilde{\mathbf{U}}$ to approximate the nonlinear function along any new trajectory, and repeatedly use  the unchanged interpolation matrix $\tilde{P}=\bar{P}\triangleq P$. To this end, we need to  evaluate the DEIM approximation quality of nonliear function by a residual, and use the residual to construct a online adaptive basis.

 The traditional  DEIM interpolates $\mathbf{f}(\bar{\mathbf{y}}(t))$ at the interpolation indexes, and  we have a zero-valued residual
\[
\mathbf{r}(\bar{\mathbf{y}}(t))=\bar{\mathbf{U}}(P^T\bar{\mathbf{U}})^{-1}P^T\mathbf{f}(\bar{\mathbf{y}}(t))-\mathbf{f}(\bar{\mathbf{y}}(t))
\]
at the interpolation points. However, for a new trajectory, the residual
\[
\mathbf{r}(\tilde{\mathbf{y}}(t))=\bar{\mathbf{U}}(P^T\bar{\mathbf{U}})^{-1}\\P^T\mathbf{f}(\tilde{\mathbf{y}}(t))-\mathbf{f}(\tilde{\mathbf{y}}(t))
\]
is not zero any more, i.e.,
\[
\|P^T\mathbf{r}(\tilde{\mathbf{y}}(t))\|_2>0.
\]

 As in section \ref{deim}, we have
\[
\mathbf{c}(\tilde{\mathbf{y}}(t))=(P^T\bar{\mathbf{U}})^{-1}P^T\mathbf{f}(\tilde{\mathbf{y}}(t)),
\]
and  $\mathbf{C}\in \mathbb{R}^{m\times M}$ is the coefficient matrix with the coefficients $\mathbf{c}(\tilde{\mathbf{y}}(t_1)),\mathbf{c}(\tilde{\mathbf{y}}(t_2)),\cdots,\\\mathbf{c}(\tilde{\mathbf{y}}(t_M))$ as columns.
Then we update the basis $\tilde{\mathbf{U}}=\bar{\mathbf{U}}+\Delta \mathbf{U}$ such that the updated basis minimizes  Frobenius norm of the residual at the observation time levels, i.e.,
\begin{equation}\label{updatedU}
\tilde{\mathbf{U}}=\arg\min_{\mathbf{U}} \|P^T(\mathbf{U}\mathbf{C}-\mathbf{F})\|_F^2,
\end{equation}
where $\mathbf{F}=[\mathbf{f}(\tilde{\mathbf{y}}(t_1)),\cdots,\mathbf{f}(\tilde{\mathbf{y}}(t_M))]\in \mathbb{R}^{n\times M}$ is  the evaluation of the nonlinear function along the new trajectory.
 Similar to Section 4.2, only $P^T\mathbf{F}\in \mathbb{R}^{m\times M}$ is required in (\ref{updatedU}),  not the complete matrix $\mathbf{F}\in \mathbb{R}^{n\times M}$. Let
  the approximation residual matrix $\mathbf{R}=\bar{\mathbf{U}}\mathbf{C}-\mathbf{F}$.  Then we solve the minimization problem
\begin{equation}\label{opm}
\arg\min_{\Delta\mathbf{U}\in \mathbb{R}^{n\times m}} \|P^T\mathbf{R}+P^T\Delta\mathbf{U}\mathbf{C}\|_F^2.
\end{equation}
  We note that the unknown matrix $\Delta \mathbf{U}$ is both premultiplied and postmultiplied.  The optimization problem (\ref{opm}) is technically tricky. To overcome the difficulty,  we decompose it into two  optimization problems, and use the Moore-Penrose pseudo-inverse to solve it approximately,
\begin{equation}\label{de1}
\arg\min_{\mathbf{b}\in \mathbb{R}^{n\times M}} \|P^T\mathbf{b}+P^T\mathbf{R}\|_F^2,
\end{equation}
\begin{equation}\label{de2}
\arg\min_{\Delta\mathbf{U}\in \mathbb{R}^{n\times m}}\|\Delta\mathbf{U}\mathbf{C}-\mathbf{b}\|_F^2.
\end{equation}
Since the interpolation matrix is a permutation matrix, which is invertible, we have $\mathbf{b}=-\mathbf{R}$ from (\ref{de1}). After taking  a transposition in (\ref{de2}),
we have
\[
\Delta\mathbf{U}=\mathbf{b}\left((\mathbf{C}^T)^{\dag}\right)^T=-\mathbf{R}\mathbf{C}^T(\mathbf{C}\mathbf{C}^T)^{-1}.
\]
This  is an approximation of (\ref{opm}).

\par Finally, We end this section with Algorithm \ref{Alg-Online}, where we give the main  steps of the stochastic online DEIM.
\begin{algorithm}[H]
	\caption{Stochastic Online DEIM} \label{Alg-Online}
	\raggedright{\bf Input:}  
	nonlinear snapshots along multiple realizations $\{\mathbf{f}(\mathbf{y}^i(t_1)),\cdots,\mathbf{f}(\mathbf{y}^i(t_M))\}_{i=1}^k \subset \mathbb{R}^{n}$ \\                           
	\raggedright{\bf Output:}
	Online updated basis matrix $\tilde{\mathbf{U}}$, interpolation matrix $P$, reduced approximation $\mathbf{f}$ \\
\raggedright \textbf{1}: Get mean snapshots matrix $[\mathbf{f}(\bar{\mathbf{y}}(t_1)),\cdots,\mathbf{f}(\bar{\mathbf{y}}(t_M))]$, \\
			where $\mathbf{f}(\bar{\mathbf{y}}(t_i))= \mathbf{f}(\frac{1}{k}\sum_{j=1}^k \mathbf{y}^j(t_i))$,  $i=1,\cdots,M$\\
\raggedright \textbf{2}: Apply DEIM to mean snapshots matrix, and get ($\bar{\mathbf{U}},\bar{P}$)\\
\raggedright \textbf{3}: Compute coefficient matrix $\mathbf{C}$, with columns  $\{\mathbf{c}(\tilde{\mathbf{y}}(t_i))=(\bar{P}^T\bar{\mathbf{U}})^{-1}\bar{P}^T\mathbf{f}(\tilde{\mathbf{y}}(t_i))\}_{i=1}^M$\\
\raggedright \textbf{4}:	 Compute residual matrix $\mathbf{R}=\bar{\mathbf{U}}\mathbf{C}-\mathbf{F}$, where $\mathbf{F}=[\mathbf{f}(\tilde{\mathbf{y}}(t_1)),\cdots,\mathbf{f}(\tilde{\mathbf{y}}(t_M))]$\\
\raggedright \textbf{5}: Solve optimization problem and have $\Delta\mathbf{U}=-\mathbf{R}\left((\mathbf{C}^T)^{\dag}\right)^T$\\
\raggedright \textbf{6}: Update basis matrix $\tilde{\mathbf{U}}=\bar{\mathbf{U}}+\Delta\mathbf{U}$, and remain interpolation matrix $P=\bar{P}$ unchanged\\
\raggedright \textbf{7}: Approximate $\mathbf{f}(\tilde{\mathbf{y}}(t))\approx \tilde{\mathbf{U}}(P^T\tilde{\mathbf{U}})^{-1}P^T\mathbf{f}(\tilde{\mathbf{y}}(t))$
\end{algorithm}


\section{Numerical results}\label{num}
In this section, we present several numerical examples to verify the previous analysis and demonstrate the efficacy of  stochastic online model reduction approach.  In the numerical simulations of this paper, we will not focus on the coarsening effect in CEM-GMsFEM because  there exist many  literatures \cite{Eric2017CEM, Mengnan2019} that have extensively addressed it. In Section \ref{ex1},  we compare the behaviours of CEM-GMsFEM in deterministic system and stochastic system. In Section \ref{ex2}, some different methods are applied to an SPDE  and show the advantage  of Online DEIM over other variants of DEIM.
A coupled stochastic porous media dynamical system is considered in Section \ref{ex3}   to illustrate the applicability of Online DEIM-MS in coupled systems.

For the numerical examples in this section, the computational spatial domain is $D=[0,1]\times [0,1]$.  We divide the domain $D$ into a $100\times 100$ fine grid, and the coarse grid is  $10\times 10$.  We take  number of local basis functions $L_i=4$ and oversampling layers $m=4$ by the empirical numerical results in  \cite{Mengnan2019}.  We will keep  the same setting for $L_i,\, m$ in the  simulations.
The solution with FEM on the fine grid is used  as the  reference solution. For simplicity of presentation, we will use MS solution
  to denote  CEM-GMsFEM solution in  figures and tables.

\subsection{Comparison in a deterministic system  and a stochastic system}\label{ex1}
Firstly, we will illustrate  the difference of the behaviours of CEM-GMsFEM between a deterministic system and
a stochastic system. They are specified as following,
\begin{equation}
(d)\left\{
\begin{aligned}
\frac{\partial u}{\partial t} - \nabla \cdot \big(\kappa(x) \nabla u\big)&=10\cos(u) \quad\quad\quad\quad \quad\quad in\quad D \times (0,T] ,\\
u(x,t)&=0 \quad\quad\quad\quad\quad\quad\quad\quad\quad\quad\, on\quad \partial D\times (0,T],\\
u(x,0)&=2\pi \sin(2\pi x_1)\sin(2\pi x_2) \quad\, in\quad D,
\end{aligned}
\right.
\end{equation}
\begin{equation}
(s)\left\{
\begin{aligned}
du - \nabla \cdot \big(\kappa(x) \nabla u\big)dt&=10\cos(u)dt+(u^2+10)dW(t) \quad in\quad D \times (0,T] ,\\
u(x,t)&=0 \quad\quad\quad\quad\quad\quad\quad\quad\quad\quad\quad\quad\quad\, on\quad \partial D\times (0,T],\\
u(x,0)&=2\pi \sin(2\pi x_1)\sin(2\pi x_2) \quad\quad\quad\quad in\quad D.
\end{aligned}
\right.
\end{equation}
The multiscale coefficient $\kappa(x)$ in the above equations is depicted in Figure \ref{coef}.
\begin{figure}[hbtp]
	\begin{minipage}[t]{0.5\linewidth}
		\centering
		\includegraphics[width=2.8in, height=2.3in]{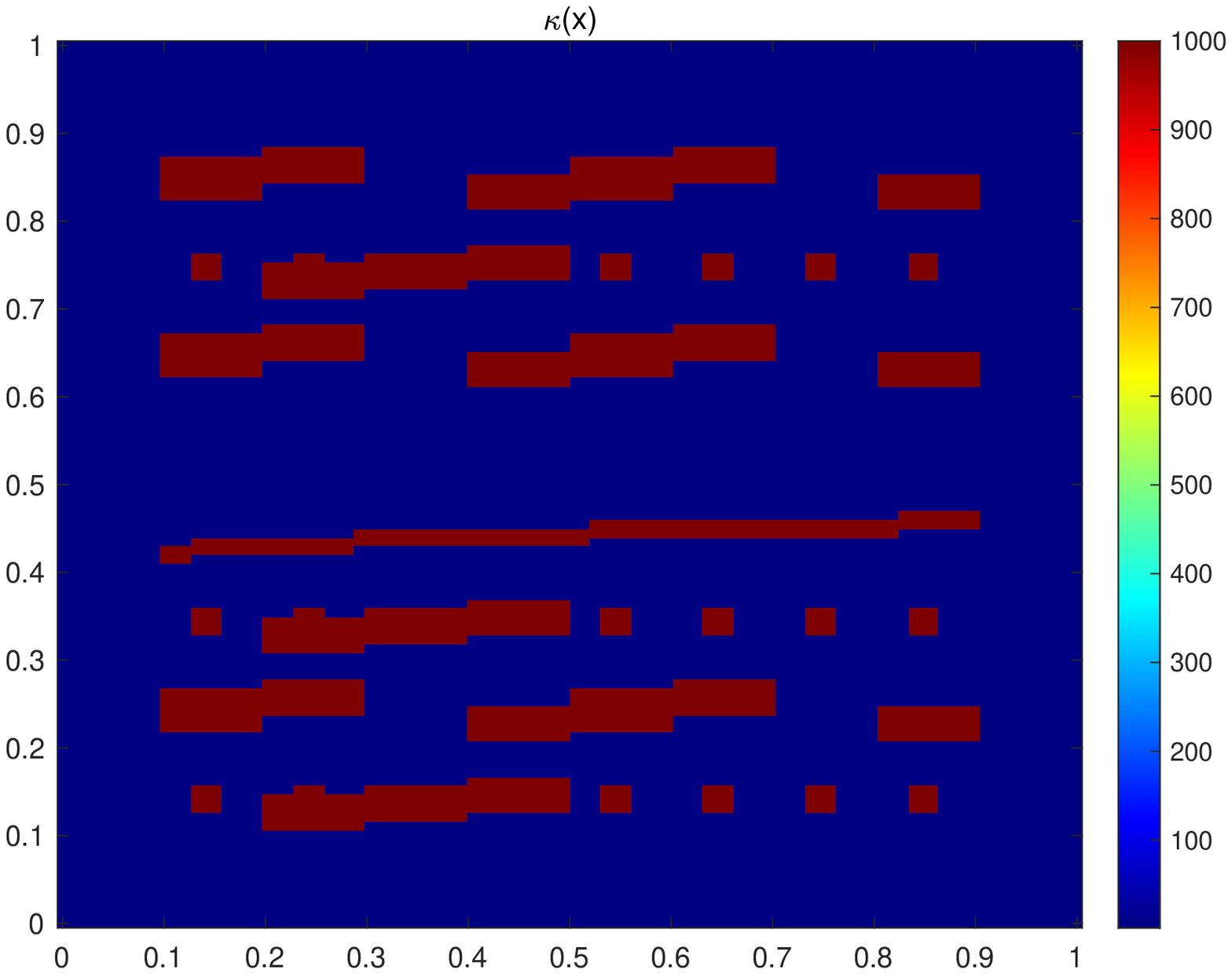}
		\caption{Multiscale coefficient $\kappa(x)$ }\label{coef}
	\end{minipage}
	\begin{minipage}[t]{0.5\linewidth}
		\centering
		\includegraphics[width=2.8in, height=2.3in]{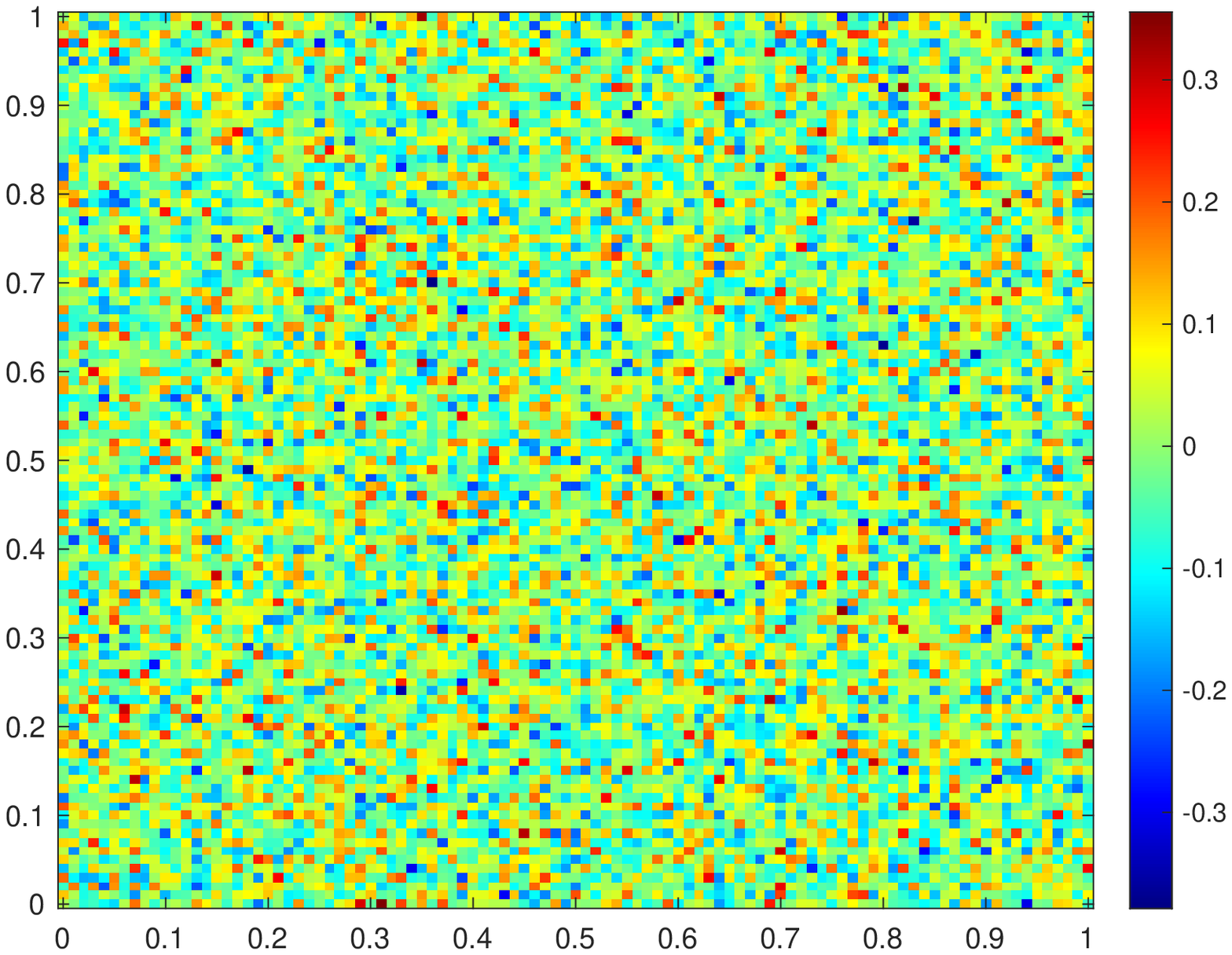}
		\caption{Brownian motion increment}\label{dw_bm}
	\end{minipage}
\end{figure}

   For the stability of numerical method, we choose the stochastic full-implicit scheme 
to treat the  temporal discretization. Then   the full discretization reads as
\begin{equation*}\label{implicit scheme}
u^{k+1}_{ms}=u^{k}_{ms}+A_{\kappa}u^{k+1}_{ms}\Delta t_k+f(u^{k+1}_{ms})\Delta t_k+g(u^k_{ms})\Delta W_k.
\end{equation*}
 The temporal is partitioned uniformly   by the interval $\Delta t=0.01$. We consider the Brownian motion with three different covariance coefficients  $Q=0.01,1,100$ as three degenerated $\mathcal{Q}$-Wiener processes. Figure \ref{dw_bm} shows   a sample of Brownian motion increment  with time interval $\Delta t$ and covariance  coefficients $Q=1$. Since the drift term $\cos(u)$ is nonlinear, we get a nonlinear equation of the unknown state $u_{ms}^{k+1}$, and use  Newton iteration to solve the nonlinear system.
Because we focus on the performance of CEM-GMsFEM, we do not use any DEIM to reduce the nonlinear model in this subsection.
 We will compute  the CEM-GMsFEM solutions of the deterministic and stochastic systems with the three different coefficients  $Q$.

\begin{figure}[htbp]
	\centering
	\includegraphics[width=5.5in]{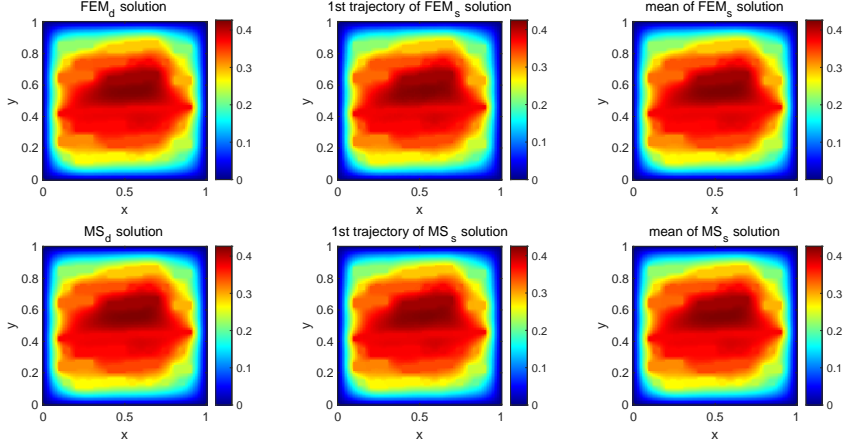}
	\caption{\textit{Solution profiles at $t=1$. Top left: FEM reference solution of deterministic system; Top middle: the first trajectory of FEM reference solution  of stochastic system with covariance coefficient $Q=0.01$; Top right: the mean of FEM reference solution  of stochastic system; Bottom: Corresponding CEM-GMsFEM solutions. }}\label{sol_det_stoQ001}
\end{figure}

 Figure \ref{sol_det_stoQ001} shows solutions of deterministic system and stochastic system at $t=1$. For the stochastic system,  mean of solutions is approximated by averaging  100 stochastic realizations. We notice that  individual behaviour (middle) and mean behaviour (right) are almost same because of a small {\color{black}coefficient $Q$} in the stochastic system. Their solution profiles  are almost the  same as the solution of deterministic system.
\begin{figure}[htbp]
	\centering
	\includegraphics[width=4.2in, height=3in]{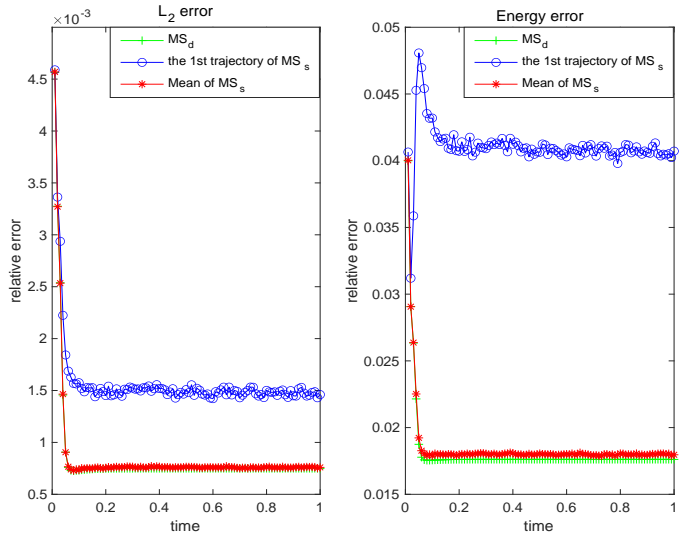}
	\caption{\textit{ Relative $L_2$  errors (left) and  relative energy  errors (right) for  MS solution (green $+$) of deterministic problem, MS solution mean (red $*$) and one trajectory of MS solution (blue $o$), $Q=0.01$.}}\label{err_det_stoQ001}
\end{figure}

 To evaluate the approximation accuracy, Figure \ref{err_det_stoQ001} shows the   relative $L_2$ error and  relative energy  error of CEM-GMsFEM solutions. Here we take the 100 stochastic realizations to approximate expectation of stochastic solutions and reference solution is the   FEM solution in fine grid.   From this figure, we see that error of the individual  trajectory of CEM-GMsFEM solution is much larger than the other two errors: $\text{MS}_d$ and mean of $\text{MS}_s$. This shows  that randomness of individual path may impact on the approximation  of CEM-GMsFEM because  we do not consider stochastic term when we construct the basis functions of CEM-GMsFEM.  Moreover, $L_2$ error of
mean of MS$_s$ is   almost the same as  the error of MS$_d$, while energy error of mean of MS$_s$  is slightly larger than MS$_d$.

\begin{figure}[htbp]
	\centering
	\includegraphics[width=5.5in]{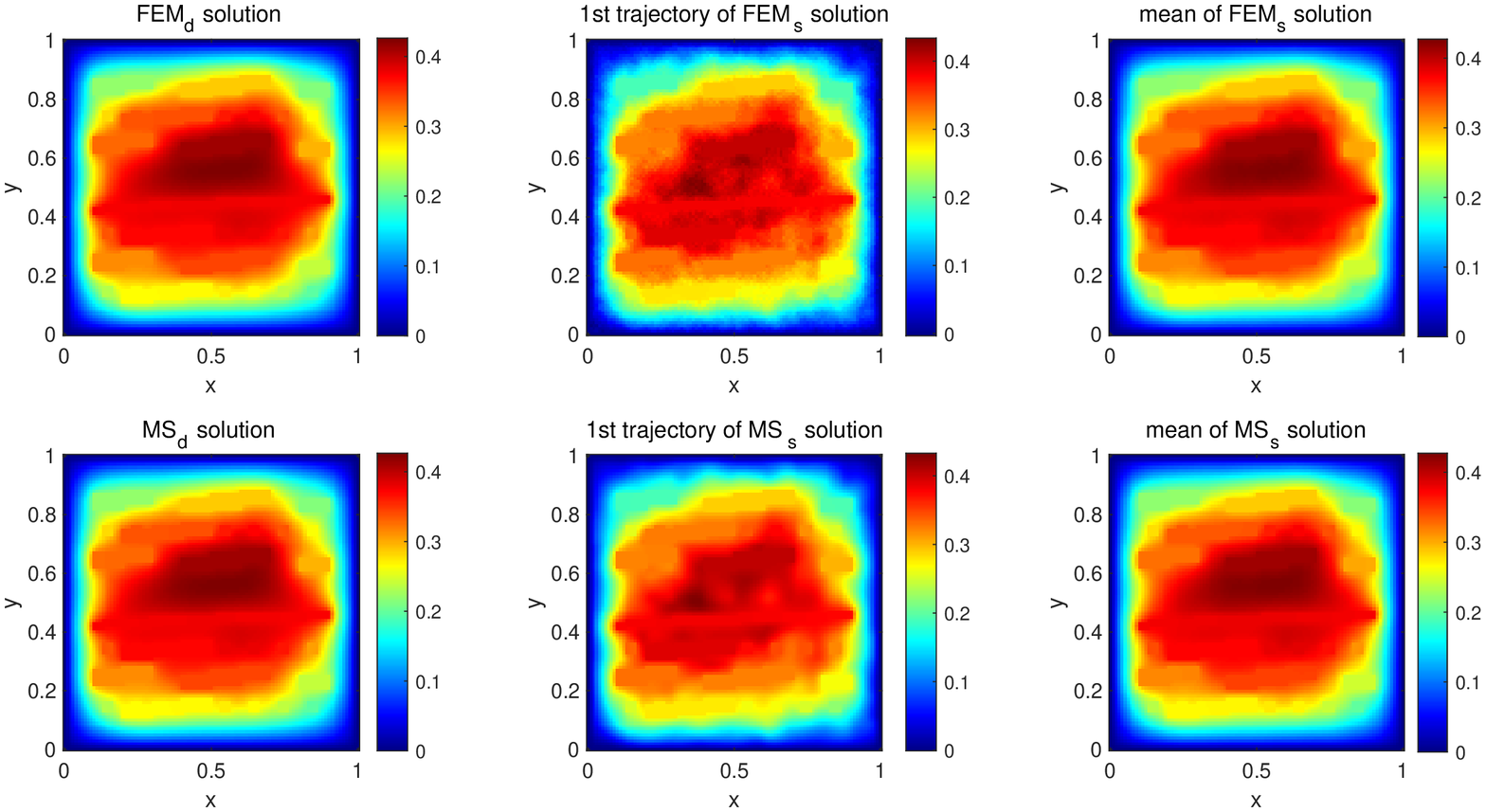}
	\caption{\textit{Solution profiles at $t=1$. Top left: FEM reference solution of deterministic system; Top middle: the first trajectory of FEM reference solution  of stochastic system with covariance coefficient $Q=1$; Top right: the mean of FEM reference solution  of stochastic system; Bottom: Corresponding CEM-GMsFEM solutions.}}\label{sol_det_stoQ1}
\end{figure}

 When the covariance coefficient $Q$ of Wiener-process increases, the trajectories  of SPDE may show significant difference among them.   By the middle plot in Figure \ref{sol_det_stoQ1} ($Q=1$), the first trajectory of the SPDE solution shows a random  perturbation comparing to deterministic model solution and mean of SPDE solution. However, mean of stochastic FEM solution behaves more like FEM$_d$ solution. Figure \ref{sol_det_stoQ1} ($Q=1$) shows that CEM-GMsFEM solutions
 approximate the FEM solutions in fine grid.  As $Q=0.01$ case, we illustrate $L_2$ relative error and energy relative error of CEM-GMsFEM solutions in Figure \ref{err_det_stoQ1} with FEM solutions as reference solutions. We see that, for the stochastic model, $L_2$ error and energy error of the first path are both much larger than deterministic CEM-GMsFEM solution. The error of mean of  CEM-GMsFEM solution to the stochastic model  is also slightly
 larger than the error of CEM-GMsFEM solution to the deterministic model. The reason may be from Monte Carlo sampling.

\begin{figure}[hbtp]
	\centering
	\includegraphics[width=4.2in, height=3in]{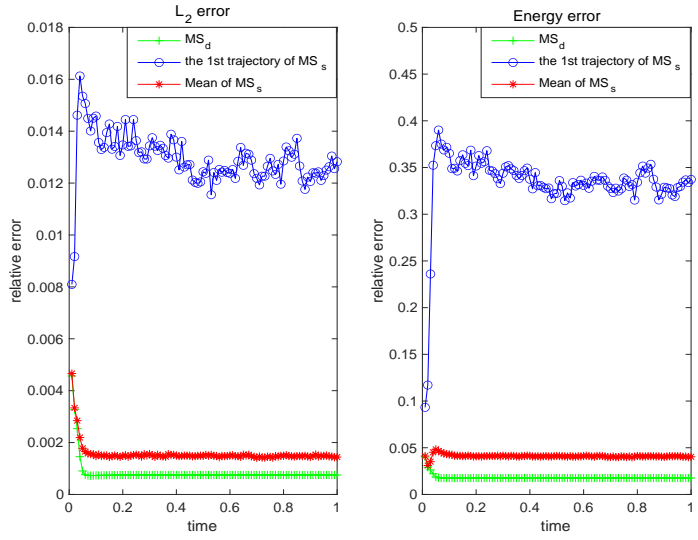}
	\caption{\textit{Relative $L_2$ errors (left) and relative energy  errors (right) for  MS solution (green $+$) of deterministic problem, MS solution mean (red $*$) and one trajectory of MS solution (blue $o$).  $Q=1$.}}\label{err_det_stoQ1}
\end{figure}

\begin{figure}[hbtp]
	\centering
	\includegraphics[width=5.5in]{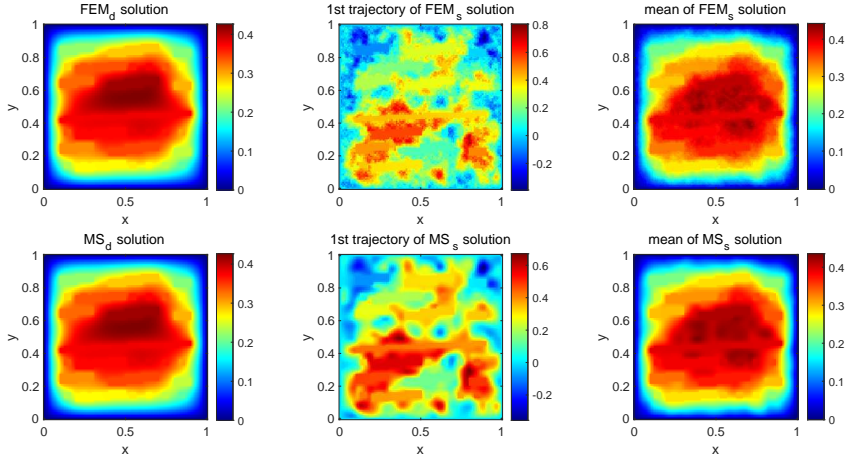}
	\caption{\textit{Solution profiles at $t=1$. Top left: FEM reference solution of deterministic system; Top middle:  the first trajectory of FEM reference solution  of stochastic system with covariance coefficient $Q=100$; Top right: the mean of FEM reference solution of stochastic system; Bottom: Corresponding CEM-GMsFEM solutions.}}\label{sol_det_stoQ10}
\end{figure}

 Comparing Figure \ref{sol_det_stoQ10} ($Q=100$)  with Figure \ref{sol_det_stoQ001} and Figure \ref{sol_det_stoQ1}, we observe that {\color{black} the larger of the covariance coefficient $Q$ leads to  less accuracy  of CEM-GMsFEM.} This is because we construct  the multiscale basis functions   without considering the stochastic influence, i.e., the stochastic perturbation of the Wiener process.  By Figure \ref{err_det_stoQ10},  we also  see that the error of mean of CEM-GMsFEM solution is much smaller than the error of individual trajectory.

\begin{figure}[hbtp]
	\centering
	\includegraphics[width=4.2in, height=3in]{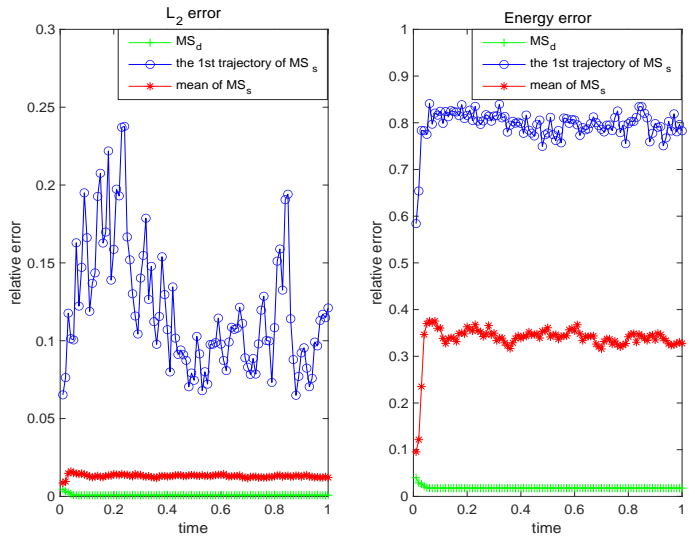}
	\caption{\textit{$L_2$ Relative $L_2$ errors (left) and  relative energy errors (right) for  MS solution (green $+$) of deterministic problem, MS solution mean (red $*$) and one realization of MS solution (blue $o$).   $Q=100$.}}\label{err_det_stoQ10}
\end{figure}


\begin{figure}[hbtp]
	\centering
	\includegraphics[width=4.5in, height=3.1in]{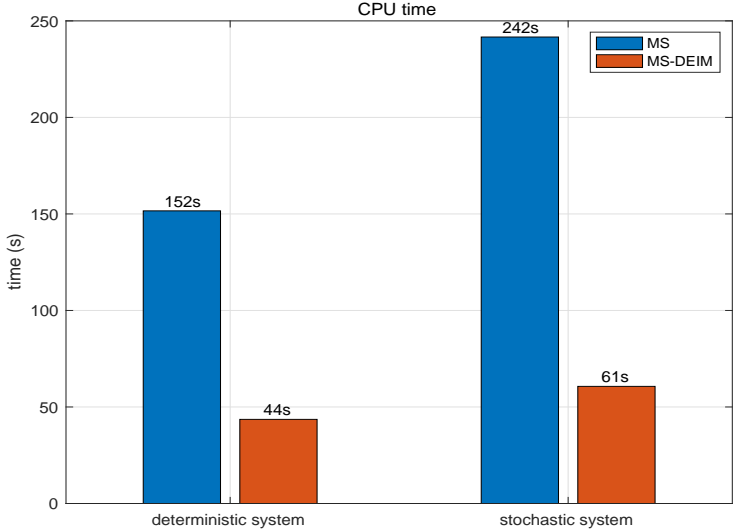}
	\caption{CPU time for the deterministic model and the stochastic model.}	\label{exm-his1}
	\label{fig:my_label}
\end{figure}

Moreover, we compare the CPU time of the different methods.  For the nonlinear terms, we use two different models for simulation: full-order model and DEIM reduced model.
	Figure \ref{exm-his1} illustrates the CUP time for deterministic model and stochastic model using two different methods: CEM-GMsFEM (MS) and CEM-GMsFEM-DEIM (MS-DEIM).
By the figure, we find that the CUP time for stochastic model is longer than the deterministic model. This is because that extra computation  is required for the diffusion term of white noise.
 The figure also clearly shows that DEIM take much less time than Newton method for both deterministic nonlinear model and stochastic nonlinear model.


\subsection{Stochastic online DEIM with CEM-GMsFEM}\label{ex2}
In this subsection, we will focus on the stochastic online DEIM incorporating with CEM-GMsFE method. The approach is referred to as Online DEIM-MS for abbreviation.
 We concentrate on two important problems about the approach: effect of snapshots in DEIM and performance of individual stochastic  realization.

\subsubsection{Effect of snapshot in offline phase and  online phase of DEIM}\label{ex21}
Stochastic online DEIM needs to use snapshots from offline phase and online phase: the offline snapshots consist of  mean information, and online snapshots consist of  new trajectory information. Thus different choices of snapshots in offline and online phase produce the different variants of online DEIM.

In this {\color{black}subsection},  we compute the relative errors from the online DEIM using different snapshots and explore the impact of the offline snapshots and online snapshots.
 Meanwhile, FEM solution and CEM-GMsFEM solution are also computed to assess the effectiveness  of online DEIM-MS. For numerical simulation, we solve the following SPDE,
\begin{equation}\label{ex21-eq}
\left\{
\begin{aligned}
du - \nabla \cdot \big(\kappa(x) \nabla u\big)dt&=\cos(u)dt+(u^2+2)dW(x,t) \quad in\quad D \times (0,T] ,\\
u(x,t)&=0 \quad\quad\quad\quad\quad\quad\quad\quad\quad on\quad \partial D\times (0,T],\\
u(x,0)&=\sin(2\pi x_1)\sin(2\pi x_2) \quad\, in\quad D,
\end{aligned}
\right.
\end{equation}
where the temporal partition and multiscale coefficient $\kappa(x)$ are chosen same as the  example \ref{ex1}.

\begin{figure}[H]
	\centering
	\includegraphics[width=3.7in, height=2.7in]{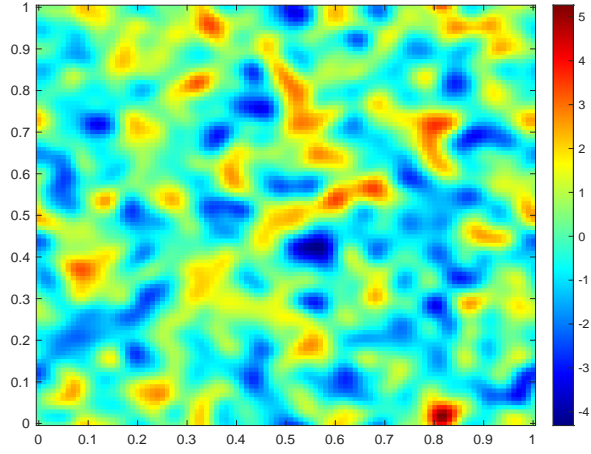}
	\caption{A sample of the $\mathcal{Q}$-Wiener process increment $\Delta W^J(x,t)$}
	\label{dw_Qwiener}
\end{figure}

 In this subsection, we use a space-time  $\mathcal{Q}$-Wiener process described in  \cite{QWiener1}.
The  covariance operator  $\mathcal{Q}$ has the  eigenfunctions
  \[
e_{j_1,j_2}(x)=\frac{1}{\sqrt{a_1a_2}}\text{e}^{2\pi i j_1x_1/a_1}\text{e}^{2\pi i j_2 x_2/a_2}
\]
 and eigenvalues ${\color{black}\mu}_{j_1,j_2}=\text{e}^{-\alpha \gamma_{j_1,j_2}},\, \alpha >0 $ and $\gamma_{j_1,j_2}=j_1^2+j_2^2$.  The approximation of $W(x, t)$ can be defined by
\begin{equation}\label{kle-tru}
W^J(x,t):=\sum_{j_1=-J_1/2+1}^{J_1/2}\sum_{j_2=-J_2/2+1}^{J_2/2}\sqrt{{\color{black}\mu}_{j_1,j_2}}e_{j_1,j_2}(x)\beta_{j_1,j_2}(t).
\end{equation}
We illustrate a sample of $\mathcal{Q}$-Wiener process increment with $J_1=J_2=100$, $\alpha=0.0005$ in Figure \ref{dw_Qwiener}.

\begin{figure}
	\begin{minipage}[t]{1\linewidth}
		\centering
		\includegraphics[width=4.5in]{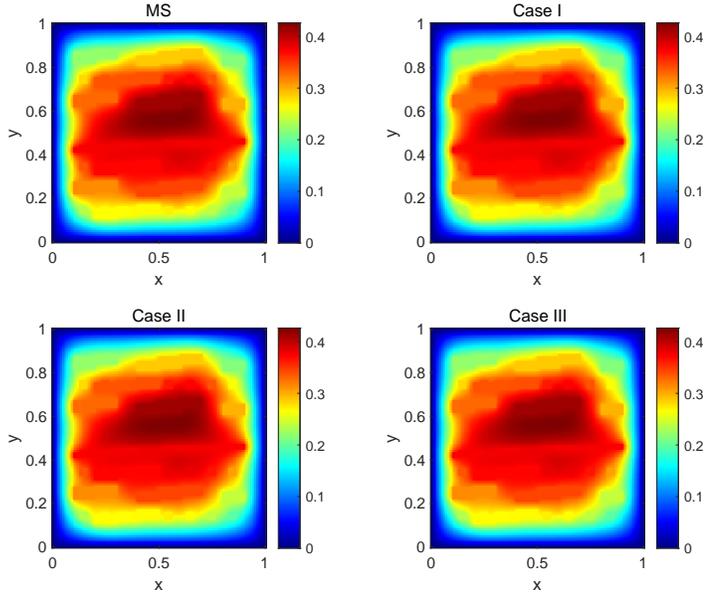}
		\caption{\textit{Solution profiles at time $t=1$. Left top:  CEM-GMsFEM solution; Right top: Case I; Left bottom: Case II; Right bottom: Case III.}}\label{compared solution}
	\end{minipage}
\end{figure}

For offline snapshots in DEIM, we compute $100$  solutions of  trajectories  to get mean information.
Here, we compare three different cases of  offline-online.
 In Case I, we take mean of nonlinear evaluation in the first half time interval, i.e.,
\[
f(\bar{u}(x,t)),\, g(\bar{u}(x,t)),\, (x,t)\in D\times[0,T/2],
\]
as the mean snapshots in the offline phase, and take the nonlinear evaluation of new realization in the first half interval as new data to update the stochastic online DEIM basis matrix in the online phase. In Case II, mean of evaluation of nonlinear functions in whole time interval, i.e. $f(\bar{u}(x,t)),\, g(\bar{u}(x,t)),\, (x,t)\in D\times[0,T]$, is used in the offline phase, and online phase is the same as Case I.  In Case III, offline procedure is the same as  Case I, but we use the  nonlinear evaluation of new trajectory in the whole time interval to update online DEIM basis in online phase. And for case III, we keep the same number of snapshots  as  Case I in online phase by selecting snapshots sparsely.
Thus, for Case \uppercase\expandafter{\romannumeral1}
and Case \uppercase\expandafter{\romannumeral2},
 solutions of a new sample in the first half temporal interval demonstrate good  fitness of these two techniques, while models' predication ability are showed by the accuracy of solution in the domain  $D\times(T/2,T]$.

Figure \ref{compared solution} depicts CEM-GMsFEM solution and the  Online DEIM-MS solutions  associated to Case I, Case II and Case III.
By the figure,  the solution profiles look  the same, and   Online DEIM-MS  is able to give a good approximation.

\begin{figure}[htbp]
	\centering
	\includegraphics[width=4.2in, height=2.9in]{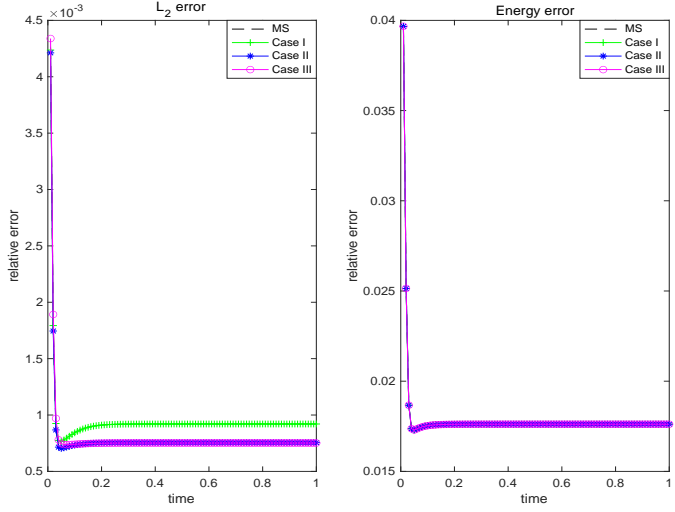}
	\caption{\textit{Relative $L_2$ errors (left) and  relative energy  errors (right) versus time $t$. }}\label{compared error}
\end{figure}

 The $L_2$ errors and energy errors of above  methods are showed in Figure \ref{compared error}.    From the figure, we see that   the $L_2$ errors are  almost less than $0.1\%$,  and  the $L_2$ error of Case \uppercase\expandafter{\romannumeral1} is the largest and the other two are smaller. This  indicates that the use of longer time information as snapshots  leads to more accurate  solution.
 Besides, we find that CEM-GMsFEM using Newton method has the smallest error. This is consistent with our expectation.
The figure shows that the errors tend to be stable as time increases. This may be because the system becomes  stationary gradually.


\subsubsection{Individual performance}\label{ex22}

In this numerical example, we are going to show the individual performance of online DEIM by presenting some new trajectories. We choose the same multiscale coefficient $\kappa$ and  $\mathcal{Q}$-Wiener process as example \ref{ex1} with noise coefficient  $Q=1,100$.  In the example, we set the nonlinear functions $f(u)=2\pi \cos(u),\, g(u)=u^2+2$ and the initial condition $u(x,0)=10\sin(2\pi x_1)\sin(2\pi x_2)$.
We solve the SPDE  using  Online DEIM  (Case \uppercase\expandafter{\romannumeral1} in {\color{black}Subsection \ref{ex21}}) combining with CEM-GMsFEM.
FEM solution  in fine grid  is  the reference solution.

%
\begin{figure}[hbtp]
	\centering
	\includegraphics[width=5.5in]{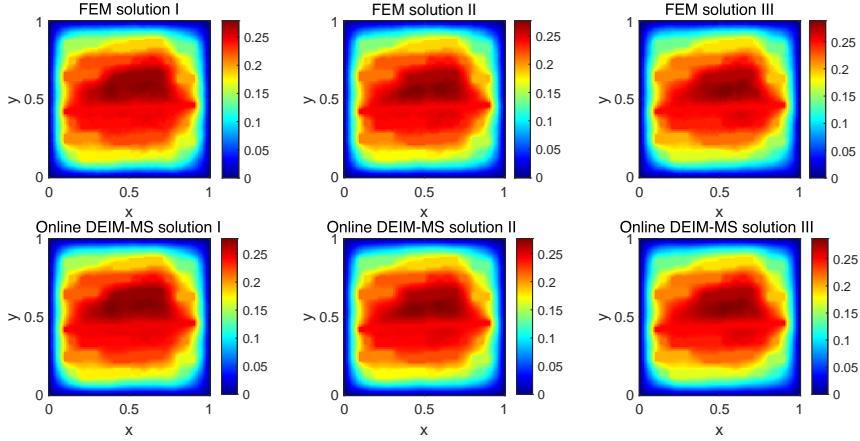}
	\caption{\textit{ Solution profiles for three arbitrary trajectories,  $Q=1$. Top: FEM reference solutions; Bottom: Online DEIM-MS solutions (\uppercase\expandafter{\romannumeral1} \uppercase\expandafter{\romannumeral2} \uppercase\expandafter{\romannumeral3} are the three different trajectories). }}\label{new8sol_Q1}
\end{figure}

 Figure \ref{new8sol_Q1} shows the solution profiles  of three different  stochastic realizations when covarience  $Q=1$. The profiles of different trajectories  are similar because of small covariance coefficient  $Q$. The figure clearly shows that the  individual performance of Online DEIM-MS solution approximates well reference solution for each  trajectory.

\begin{figure}[htbp]
	\centering
	\includegraphics[width=5.5in]{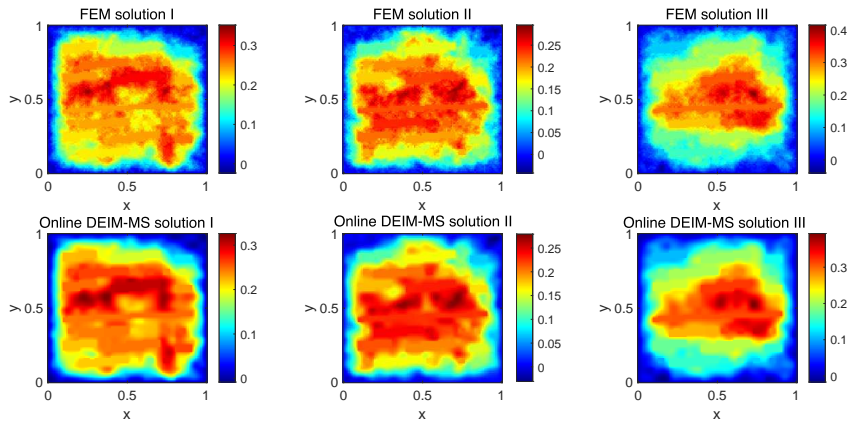}
	\caption{\textit{ Solution profiles for three arbitrary trajectories,  $Q=100$.. Top: FEM reference solutions; Bottom: Online DEIM-MS solutions (\uppercase\expandafter{\romannumeral1} \uppercase\expandafter{\romannumeral2} \uppercase\expandafter{\romannumeral3} are three different trajectories). }}\label{new8sol_Q10}
\end{figure}

Figure \ref{new8sol_Q10} depicts the solution profiles when $Q=100$.   Because the diffusion covarience is  large,  individual profiles  are quite different from the three
different realizations.
  But the Online DEIM-MS solution looks very like the FEM solutions in general. That is to say, the individual trajectory performs accurately even though the diffusion covarience  becomes  large.

\begin{figure}[htbp]
	\centering
	\includegraphics[width=4.3in, height=3in]{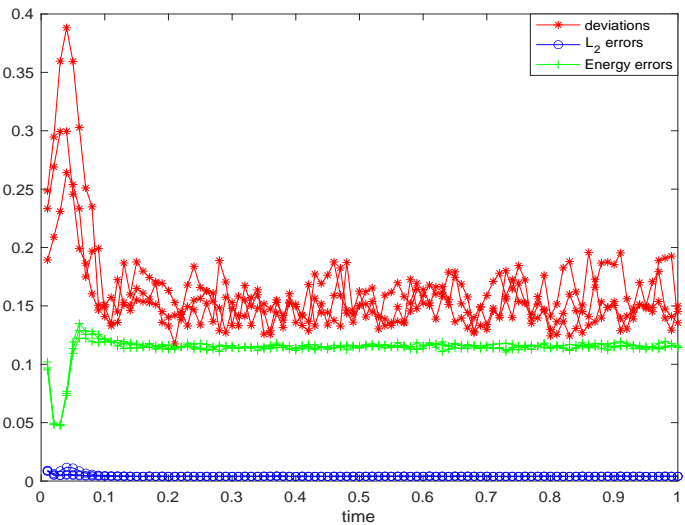}
	\caption{\textit{Relative $L_2$ errors (blue o), relative energy  errors (green +) and derivations for the three arbitrary trajectories, $Q=100$.  }}\label{new8err_Q1}
\end{figure}

 In Figure \ref{new8err_Q1},  we illustrate $L_2$ errors, energy errors and deviations  of the arbitrary  three  realizations.
Here  the deviations are given by
\[
\text{deviation}=\sqrt{\frac{\|\text{FEM solution of arbitrary realization }-\text{ mean of FEM solutions}\|_2}{\|\text{mean of FEM solutions}\|_2}}.
\]
Large deviation means that the problem has a very  different new realization from the mean information. It is obvious in Figure  \ref{new8err_Q1} that the deviations are much larger than the errors. Thus large deviations and small  errors indicate the good performance of Online DEIM-MS for new trajectories.  One can predict how the  state evolves with respect to time.

{\color{black}
\par In  Table \ref{table}, we make a comparison of the approximation error and the computational cost between stochastic Online DEIM  and traditional DEIM, referred as Offline DEIM when these methods are applied to a single trajectory. Since their offline phases are the same, we only record the time of online phases. For Offline DEIM, the CPU time is the time of CEM-GMsFEM computation with offline-reduced nonlinear functions. But for Online DEIM, the CPU time  includes both the time of online update of DEIM basis matrix and the time of MS computation with online-reduced nonlinear functions. And the relative $L_2$ error and relative energy error are both listed in the table.}

\begin{table}[htbp]
	\centering
	{\color{black}
		\caption{A comparison of CPU time and relative errors.}
		\label{table} }
	\begin{tabular}{|c|c|c|c|}
		 \hline
		 Method & CPU time(s) & $L_2$ error & Energy error\\
		 \hline
		 MS & 209.7017 & $5.5858\times 10^{-3}$ & $7.3928\times 10^{-2}$\\
		 \hline
		 Offline DEIM-MS & 53.4225 & $7.3864\times 10^{-3}$ &  $7.4025\times 10^{-2}$\\
		 \hline
		 Online DEIM-MS & 53.4835 & $6.6042\times 10^{-3}$ &  $7.3973\times 10^{-2}$ \\
		 \hline
	\end{tabular}	
\end{table}

{\color{black}
Table \ref{table} shows that Online DEIM has better accuracy while it takes slightly more time than Offline DEIM. Comparing with CEM-GMsFEM without any reduction, Online DEIM is able to significantly improve the computation efficiency and keep good accuracy.
}
\subsection{A coupled  stochastic system in porous media}\label{ex3}

In this subsection, we consider a coupled system, in which SDE is  coupled to a porous media equation. In this class of system, the fluid velocity is affected by the particle deposition, and the effect is described by a stochastic differential equation where the particle velocity is influenced by the fluid velocity in turn. That is to say, the particle dynamics (i.e.,  the stochastic differential equation in (\ref{coupled})) governs the permeability change. Thus the coupled system produces  a stochastic porous media problem in the real world. Here is a specific system:

\begin{equation}\label{coupled}
\left\{
\begin{aligned}
\frac{\partial u(x,t)}{\partial t}&= \nabla\cdot \big(\kappa(x) \nabla u\big)+10e^{2v}\sin(u) \quad in\quad D \times (0,T] ,\\
dv(x,t) &=-10(v-u)dt+\frac{1}{\sqrt{0.1}}dW(t) \quad in\quad D \times (0,T] ,\\
u(x,t)&=0 \quad\quad\quad\quad\quad\quad\quad\quad\quad\quad\quad\quad\,\, on\quad \partial D\times (0,T],\\
u(x,0)&=1 \quad\quad\quad\quad\quad\quad\quad\quad\quad\quad\quad\quad\,\, in\quad D,\\
v(x,0)&=2 \quad\quad\quad\quad\quad\quad\quad\quad\quad\quad\quad\quad\,\, in\quad D,
\end{aligned}
\right.
\end{equation}
where $u$ represents the fluid velocity and $v$ is the particle velocity. Here the permeability field $\kappa(x)$ is taken as the Tenth Society of Petroleum Engineers Comparative Solution Project (SPE10), which is shown in Figure \ref{SPE10}.
\begin{figure}[htbp]
	\centering
	\includegraphics[width=4.2in, height=3in]{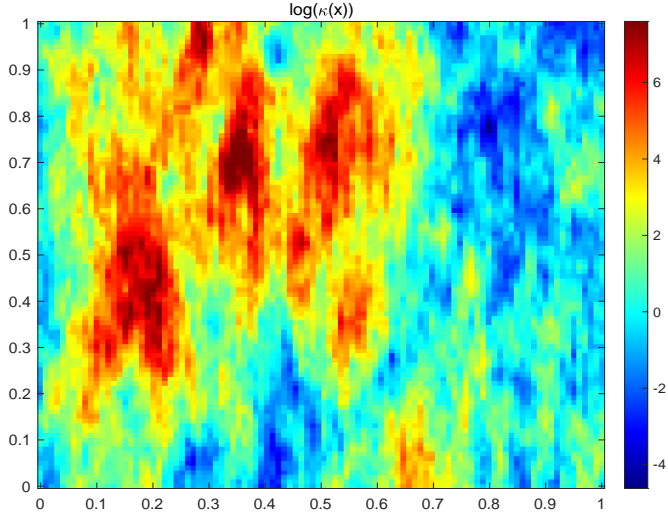}
	\caption{A permeability field (a layer from SPE10)}\label{SPE10}
\end{figure}

The model (\ref{coupled}) describes an equation with the heterogeneous diffusion $\kappa(x)$, and $e^{2v}$ characterizes  the affection by particle deposition, and  10 times of fluid velocity affects the particle velocity. We choose constant $1$ and $2$ as initial velocity of fluid and of particle respectively, because  the particle dynamics  usually occurs faster compared to fluid flow. Besides,  the long time evolution of system will cause the effect  that solution $u$ has the sharp decay to zero, and the solution becomes  small after some time.
To show the difference between  different DEIM-MS approaches, we solve the problem in domain $D=[0,1]\times [0,1]$, with time interval $t\in [0,0.1]$.
In this example, we  divide  the time into $N=100$ intervals with step size $\Delta t=0.001$, and compute $100$ trajectories' solutions to get mean snapshots, and  $Q=0.01$ for the
Wiener process $W(t)$.

 We solve the coupled system sequentially (the semilinear  PDE is solved after SDE), and we use drift-implicit Milstein scheme to solve SDE, while PDE is solved by four different methods
: MS, Offline DEIM-MS, Online DEIM-MS and FEM in fine grid.    We compute Offline DEIM-MS solution by the reduced model constructed  in the offline phase where we utilize only  the snapshots of nonlinear function $e^{2v}\sin(u)$ along mean of offline trajectories, while Online DEIM-MS solution is computed by Case \uppercase\expandafter{\romannumeral1} in Section \ref{ex2}.

\begin{figure}[hbtp]
	\centering
	\includegraphics[width=4.5in, height=3.2in]{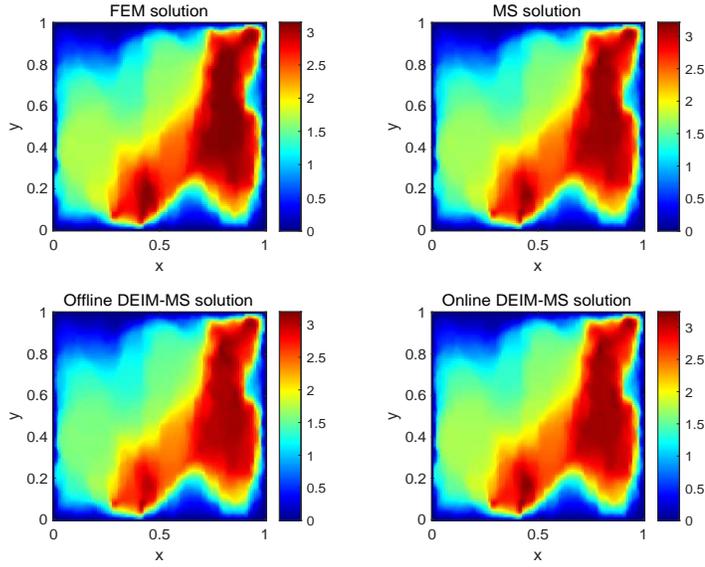}
	\caption{\textit{Solution profiles  of fluid velocity $u$}}\label{sol_u}
\end{figure}

  Figure \ref{sol_u}    shows the profiles of fluid velocity $u$.   By the figure,   Online DEIM-MS solution at time $t=0.055$ is more close to the MS and FEM solution, while Offline DEIM-MS solution  obviously differs from reference solution in some locations. This shows that  Online DEIM-MS is better than Offline DEIM-MS for  prediction.

\begin{figure}[htbp]
	\centering
	\includegraphics[width=4.2in, height=3in]{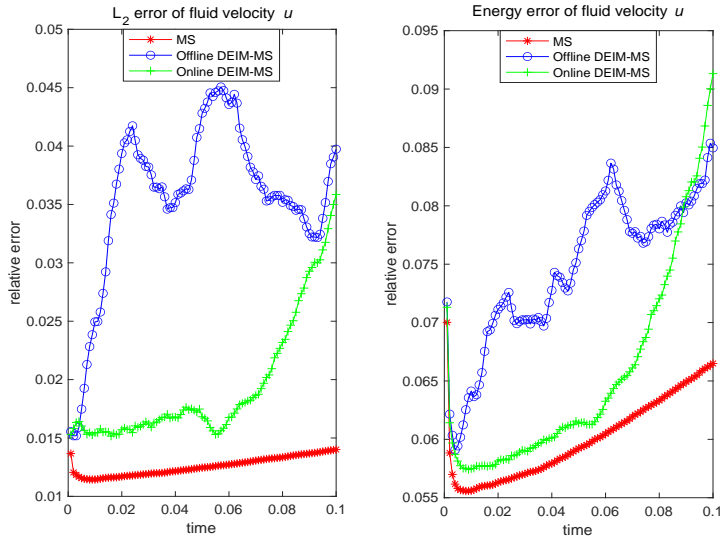}
	\caption{\textit{Relative $L_2$ errors (left) and energy errors (right) of fluid velocity $u$ versus  time $t$ }}\label{err_u}
\end{figure}
\begin{figure}[htbp]
	\centering
	\includegraphics[width=4.5in, height=3.2in]{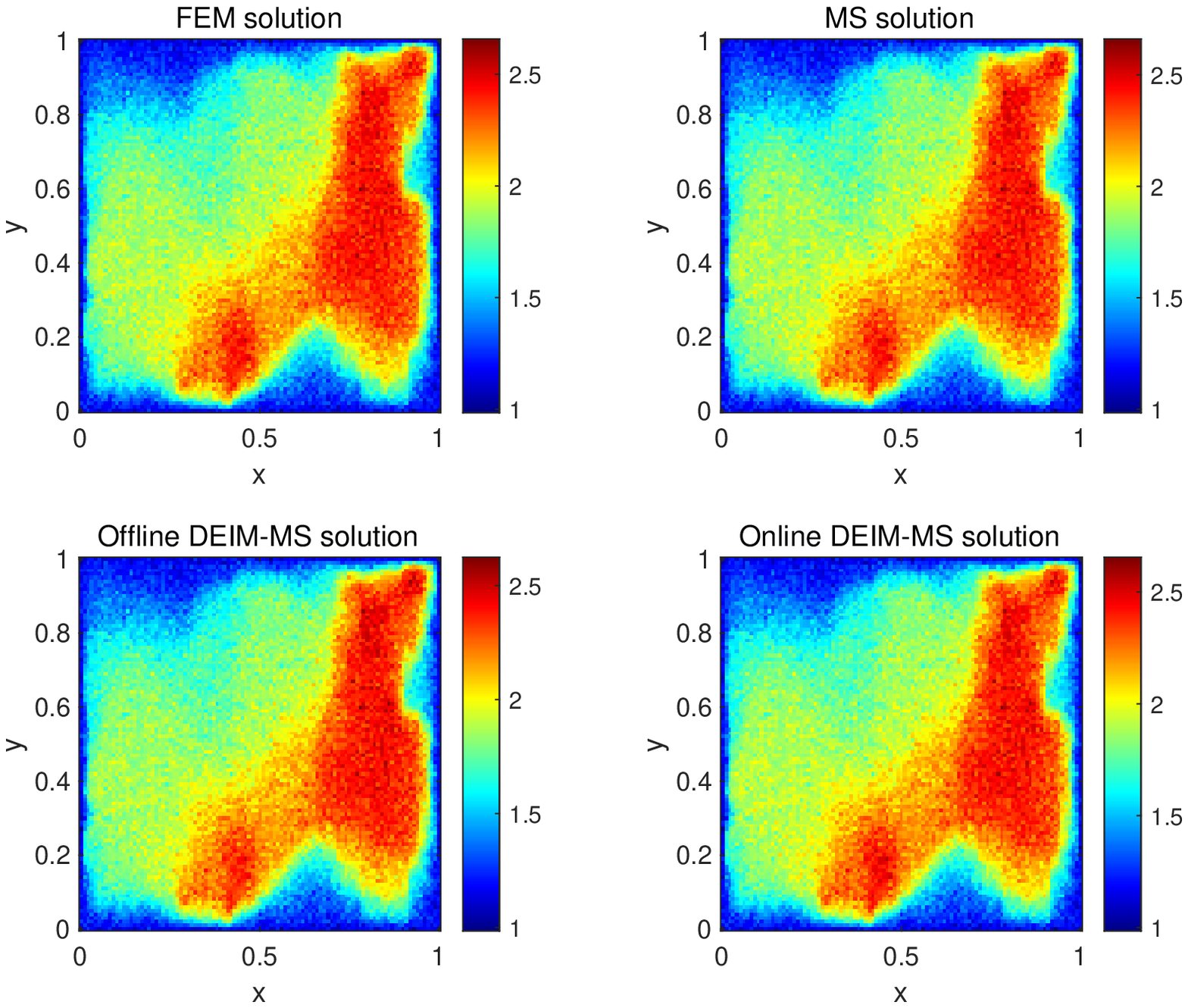}
	\caption{\textit {Solution profiles  of particle velocity $v$ at $t=0.055$.}}\label{sol_v}
\end{figure}

  Figure \ref{err_u} shows the relative errors versus time.  By the figure,  we observe that the error of MS solution is  the smallest. This  coincides with the previous examples. In the offline phase,mean information in the first half time interval is used  to construct the snapshots, compute the DEIM basis matrix and interpolation matrix. This yields a reduced model using  Offline DEIM and Online DEIM. Moreover, for Online DEIM, we also utilize the data of new path to update the basis matrix  in the online phase. Therefore, for $t\in [0,T/2)$, where $T=0.1$, Online DEIM produces the fitness of  system evolution, while for $t\in (T/2,T]$, Online DEIM can predict the propogation of system. But the prediction may perform well only in a short time interval, since the correlation will decay with the elapse of time. Comparing the error curve in Figure \ref{err_u}, we see that Online DEIM-MS is much more precise than Offline DEIM-MS for $t\in [0,T/2)$.   For $t\in (T/2,T]$, the error of Online DEIM-MS increases over time as expected. Although the error of Online DEIM-MS grows in the later half interval, it is still much less than the error of Offline DEIM-MS. This implies that  the prediction  of Online DEIM is better than Offline DEIM.  This  verifies  the accuracy of fitness and forecast of Online DEIM-MS approach.

 In Figure \ref{sol_v}, we give the profiles of particle velocity $v$ at time $t=0.055$. We note that $v$   is solved by plugging  $u$ into the stochastic equation.  The solution  $u$ is computed  by  four different methods: FEM in fine grid, MS, Offline DEIM-MS and Online DEIM-MS. Different solutions of $u$ lead to different solutions of $v$ because they are coupled each other.  By the figure,  we find  that the solution profiles   are very close to each other. Comparing with Figure \ref{sol_u},the  particle velocity $v$ is not as smooth as the fluid velocity, due to the direct impact from Brownian motion. However, the particle and the fluid velocity have the similar pattern  because the particle dynamics and the fluid dynamics interact on each other.
 We see that, for arbitrary $t\in [0,T]$, the error of Offline DEIM-MS solution is larger than that of Online DEIM-MS. The trend of error of particle velocity is consistent with the error of fluid velocity.

\begin{figure}[htbp]
	\centering
	\includegraphics[width=4.2in, height=2.6in]{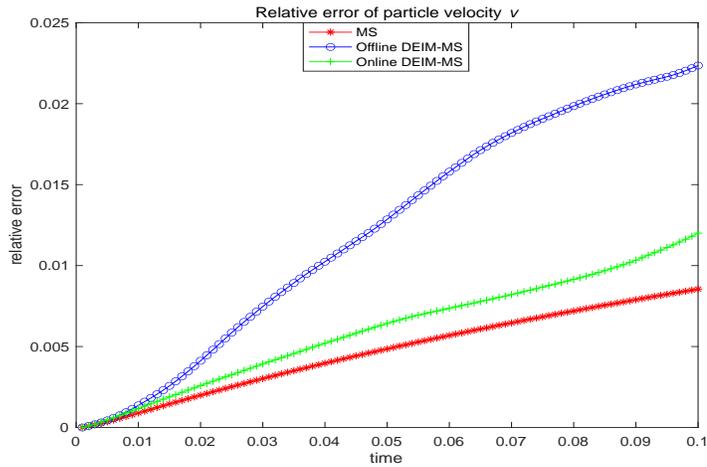}
	\caption{\textit{Relative $L_2$  errors of  the particle velocity $v$ versus  time $t$}}\label{err_v}
\end{figure}


\section{Conclusion}
We presented a stochastic online model reduction approach for multiscale nonlinear stochastic parabolic PDEs. Multiscales and nonlinearity  significantly impact on the computation using  traditional FEM
in a fine grid.  CEM-GMsFEM and DEIM can substantially improve the computation efficiency. We have carried out  the error analysis of CEM-GMsFEM for the nonlinear stochastic PDEs.
However, the semi-discrete system is still nonlinear in a high-dimensional space by CEM-GMsFEM. In order to overcome this challenge,
  we proposed the stochastic online DEIM method, which used  the offline mean information and  the online trajectory snapshots. The online DEIM significantly
  reduced the computation complexity for the nonlinear system.
 By incorporating the stochastic online DEIM  with CEM-GMsFEM,    we developed  a  stochastic multiscale  model reduction method: Online DEIM-MS.
 This method showed superiority in prediction compared with Offline DEIM approach. We presented a few numerical examples from the  porous media application to
 show effectiveness of    the   stochastic online model reduction method.


\end{document}